\documentclass[a4paper,11pt]{article}

\usepackage{graphicx}
\usepackage{amsmath}
\usepackage{amsfonts}
\usepackage{titlesec}
\usepackage{amsthm}
\usepackage{amssymb}
\usepackage{geometry}
\usepackage{color}
\usepackage{accents}

\definecolor{dark-green}{rgb}{0.0,0.4,0.0}
\definecolor{pantone_369}{rgb}{0.4784, 0.7098, 0.0863}

\usepackage{algorithm2e}

\newcommand{\ms}{\mbox{\rm \tiny ms}}

\newcommand{\Vms}{V_{H,\Omega}^{\ms}}
\newcommand{\Vmsk}{V_{H,k}^{\ms}}

\newcommand{\umsk}{u_{H,k}^{\ms}}
\newcommand{\hatumsk}{\hat{u}_{H,k}^{\ms}}

\newcommand{\dott}{(\cdot,t)}
\newcommand{\sym}{ \mbox{\tiny \rm sym} }

\newcommand{\R}{\mathbb{R}}
\newcommand{\T}{\mathcal{T}}

\newcommand{\eps}{\varepsilon}

\newcommand{\aeps}{a^\eps}
\newcommand{\ahom}{a^0}

\newcommand{\ueps}{u^\eps}
\newcommand{\veps}{v^\eps}
\newcommand{\uhom}{u^0}

\newcommand{\hatueps}{\hat{u}^\eps}

\newcommand{\beps}{b^{\eps}}

\newcommand{\uHk}{u_{H,k}}
\newcommand{\uHtk}{u_{H,\triangle t, k}}
\newcommand{\bHk}{b_{H,k}}

\newcommand{\hatuHk}{\hat{u}_{H,k}}

\newcommand*{\dt}[1]{\accentset{\mbox{\large\bfseries .}}{#1}}
\newcommand*{\ddt}[1]{\accentset{\mbox{\large\bfseries .\hspace{-0.25ex}.}}{#1}}

\newcommand{\en}{\enspace}

\newcommand{\quotes}[1]{``#1''}

\newtheorem{theorem}{Theorem}[section]

\newtheorem{lemma}[theorem]{Lemma}
\newtheorem{proposition}[theorem]{Proposition}

\newtheorem{problem}[theorem]{Problem}
\theoremstyle{definition}
\newtheorem{definition}[theorem]{Definition}
\newtheorem{remark}[theorem]{Remark}

\title{Localized orthogonal decomposition method for the wave equation with a continuum of scales}

\begin{document}
\maketitle

\begin{center}
{\large Assyr Abdulle\footnote[1]{ANMC, Section de Math\'{e}matiques, \'{E}cole polytechnique f\'{e}d\'{e}rale de Lausanne, 1015 Lausanne, Switzerland, Assyr.Abdulle@epfl.ch} and Patrick Henning\footnote[2]{Institut f\"ur Numerische und Angewandte Mathematik, Westf\"alische Wilhelms-Universit\"at M\"unster, Einsteinstr. 62, D-48149 M\"unster, Germany, Patrick.Henning@wwu.de}}\\[2em]
\end{center}

\renewcommand{\thefootnote}{\fnsymbol{footnote}}
\renewcommand{\thefootnote}{\arabic{footnote}}

\begin{abstract}
This paper is devoted to numerical approximations for the wave equation with a multiscale character. Our approach is formulated in the framework of the Localized Orthogonal Decomposition (LOD) interpreted as a numerical homogenization with an $L^2$-projection. We derive explicit convergence rates of the method in the $L^{\infty}(L^2)$-, $W^{1,\infty}(L^2)$- and $L^{\infty}(H^1)$-norms without any assumptions on higher order space regularity or scale-separation. The order of the convergence rates depends on further graded assumptions on the initial data. We also prove   the convergence of the method in the framework of G-convergence without any  structural assumptions on the initial data, i.e. without assuming that it is well-prepared. This rigorously justifies the method. Finally, the performance of the method is demonstrated in numerical experiments.
\end{abstract}



\section{Introduction}
\label{section:introduction}

This work is devoted to the linear wave equation in a heterogeneous medium with multiple highly varying length-scales. We are looking for an unknown wave function $\ueps$ that fulfills the equation
\begin{align}
\nonumber\label{wave-equation-strong} \partial_{tt} \ueps(x,t) - \nabla \cdot \left( \aeps(x) \nabla \ueps(x,t) \right) &= F(x,t) \qquad \mbox{in } \Omega \times(0,T], \\
\ueps(x,t) &= 0 \hspace{48pt} \mbox{on } \partial \Omega \times [0,T], \\
\nonumber\ueps(x,0) = f(x) \quad \mbox{and} \quad \partial_t \ueps(x,0) &= g(x) \hspace{35pt} \mbox{in } \Omega.
\end{align}
Here, $\Omega$ denotes the medium, $[0,T] \subset \mathbb{R}^+$ the relevant time interval, $\aeps$ the wave speed, $F$ a source term and $f$ and $g$ the initial conditions for the wave and its time derivative respectively. The parameter $\eps$ is an abstract parameter which simply indicates that a certain quantity is subject to rapid variations on a very fine scale (relative to the extension of $\Omega$). The parameter $\eps$ can be seen as a measure for the minimum wave length of these variations. Precisely, we assume that for a function  $z^{\eps}$,  $\| z^\eps\|_{H^s(\Omega)}$ is large for $s>1$ and cannot be approximated with a FE function on a coarse grid while the typical fine mesh needed to approximate such functions is computationally too expensive.
However, we stress that we do not assign a particular value or meaning to $\eps$ in this work. Due to the fast variations in the data functions, which take place at a scale that is very small compared to the total size of the medium, these problems are typically referred to as {\it multiscale problems}. Equations such as (\ref{wave-equation-strong}) with a multiscale character arise in various fields such as material sciences, geophysics, seismology or oceanography. For instance, the propagation and reflection of seismic waves can be used to determine the structure and constitution of subsurface formations. In particular, it is necessary in order to locate petroleum reservoirs in earth's crust.

Trying to solve the multiscale wave equation (\ref{wave-equation-strong}) with a 
direct computation, using e.g. finite elements or finite differences, exceeds typically the possibilities of today's super computers. The reason is that the computational mesh needs to resolve all variations of the coefficient matrix $\aeps$, which leads to extremely high dimensional solution spaces and hence linear systems of tremendous size that need to be solved at every time step. 

In order to tackle this issue, numerical homogenization can be applied. The term numerical homogenization refers to a wide set of numerical methods that are based on replacing the multiscale problem (\ref{wave-equation-strong}) by an effective/upscaled/homogenized equation which is of the same type as the original equation, but which has no longer multiscale properties (the fine scale is \quotes{averaged out}). Hence, it can be solved in lower dimensional spaces with reduced computational costs. The obtained approximations yield the effective macroscopic properties of $\ueps$ (i.e. they are good $L^2$-approximations of $\ueps$). Multiscale methods that were specifically designed for the wave equation, can be e.g. found in \cite{AbG11,EHR11,KoM06,OwZ08,VMK05}. In Section \ref{subsection-surbey-multiscale-methods} we give a detailed overview on these approaches.

In this paper, we will present a multiscale method for the wave equation which does neither require structural assumptions such as a scale separation nor does it require regularity assumptions on $\aeps$. We will not exploit any higher space regularity than $H^1$. Furthermore, it is not necessary to solve expensive global elliptic fine scale problems in a pre-process (sometimes referred to as the 'one-time-overhead', cf. \cite{JEG10,JiE12,OwZ08}). Our method is based on the following consideration: the $L^2$-projection $P_{L^2}$ of the (unknown) exact solution $\ueps$ into a coarse finite element space is assumed to be a good approximation to an (unknown) homogenized solution. Furthermore, the $L^2$-projection $P_{L^2}(\ueps)$ can be well approximated in a low dimensional finite element space. If we can derive an equation for $P_{L^2}(\ueps)$, all computations can be performed in the low dimensional space and are hence cheap. This approach fits into the framework of the Localized Orthogonal Decomposition (LOD) initially proposed in \cite{MaP14} and further developed in \cite{HeP13,HeM14}. The idea of the framework is to decrease the dimension of a high dimensional finite element space by splitting it into the direct sum of a low dimensional space with high $H^1$-approximation properties and a high dimensional remainder space with negligible information. The splitting is based on an orthogonal decomposition with respect to an energy scalar product. In this work we will pick up this concept, since the remainder space in the splitting is nothing but the kernel of the $L^2$-projection.

The general setting of this paper is established in Section \ref{section:motivation}, where we also motivate the method. In Section \ref{main-section-LOD} we introduce the space discretization that is required for formulating the method in a rigorous way. In Section \ref{section-main-results-and-misc} we state our main results and we give a survey on other multiscale strategies.
These main results 
are proved in Section \ref{section:proofs}. Finally, numerical experiments confirming our theoretical results are presented in Section \ref{section-numerics}.

\section{Motivation - Numerical homogenization by $L^2$-projection}
\label{section:motivation}

In the following, we consider the wave equation (\ref{wave-equation-strong}) in weak formulation, i.e. we seek $\ueps \in L^2(0,T;H^1_0(\Omega))$ and $\partial_{tt} \ueps \in L^2(0,T;H^{-1}(\Omega))$ such that for all $v \in H^1_0(\Omega)$  and a.e. $t > 0$
\begin{align}
\label{wave-equation-weak}
\nonumber\langle \partial_{tt} \ueps \dott, v \rangle + \left( \aeps \nabla \ueps \dott, \nabla v \right)_{L^2(\Omega)} &= \left( F \dott, v \right)_{L^2(\Omega)},\\
\left( \ueps( \cdot , 0 ) , v \right)_{L^2(\Omega)} &= \left( f , v \right)_{L^2(\Omega)},\\
\nonumber \left( \partial_t \ueps( \cdot , 0 ) , v \right)_{L^2(\Omega)} &= \left( g , v \right)_{L^2(\Omega)}.
\end{align}
Here, the dual pairing is understood as $\langle L, v \rangle=L(v)$ for $L\in H^{-1}(\Omega)$ and $v \in H^1_0(\Omega)$. In the following, we make use of the shorthand notation $W^{m,p}(H^s_0):=W^{m,p}(0,T;H^s_0(\Omega))$ for $0\le m < \infty$, $1\le p \le \infty$ and $0\le s \le 1$. We denote $H^0_0(\Omega):=L^2(\Omega)$ and $W^{0,p}:=L^p$.
In order to guarantee the existence of a unique solution of the system (\ref{wave-equation-strong}), we make the following assumptions:
\begin{itemize}
\item[(H0) $\bullet$] $\Omega\subset\mathbb{R}^{d}$, for $d=1,2,3$, denotes a bounded Lipschitz domain with a piecewise polygonal boundary;
\item the data functions fulfill $F \in L^2(0,T;L^2(\Omega))$, $f \in H^1_0(\Omega)$ and $g \in L^2(\Omega)$;
\item the matrix-valued function $\aeps\in [L^\infty(\Omega)]^{d\times d}_{\sym}$ that describes the propagation field is symmetric and it is uniformly bounded and positive definite, i.e. $\aeps \in \mathcal{M}(\alpha,\beta,\Omega)$ for $\beta\geq\alpha>0$. Here, we denote
\begin{eqnarray}\label{e:spectralbound}
\lefteqn{\mathcal{M}(\alpha,\beta,\Omega) :=} \\
\nonumber&\hspace{0pt}&\{ a \in [L^\infty(\Omega)]^{d\times d}_{\sym}| \hspace{5pt}
\alpha |\xi|^2 \le a(x) \xi \cdot \xi \le \beta |\xi|^2 \enspace\text{for all } \xi \in \R^d \mbox{ and almost all }x\in \Omega\}.
\end{eqnarray}
\end{itemize}
Under the assumptions listed in (H0) there exists a unique weak solution $\ueps$ of the wave equation (\ref{wave-equation-strong}) with $\partial_t \ueps \in L^2(0,T;L^2(\Omega))$. Furthermore, $\ueps$ is regular in time, in the sense that $\ueps \in C^0(0,T;H^1_0(\Omega))$ and $\partial_t \ueps \in C^0(0,T;L^2(\Omega))$. This result can be e.g. found in \cite[Chapter 3]{LiM72a}.

In addition to the above assumptions, we also implicitly assume that the wave speed $\aeps$ has rapid variations on a very fine scale which need to be resolved by an underlying fine grid. The dimension of the resulting finite element space (for the spatial discretization) is hence very large. The method proposed in the subsequent sections aims to reduce the computational cost that is associated with solving the discretized wave equation in this high dimensional finite element space.

In order to simplify the notation, we define
\begin{align}
\label{definition:problem:bilinearform}\beps(v,w):=\int_{\Omega} \aeps \nabla v \cdot \nabla w \qquad \mbox{for } v,w \in H^1_0(\Omega).
\end{align}

We next motivate a multiscale method for the wave equation and discuss the framework of our approach. All the subsequent discussion will be later rigorously justified by a general convergence proof. We are interested in finding a {\it homogenized} or {\it upscaled} approximation of $\ueps$. In engineering applications this can be a function describing the macroscopically measurable properties of $\ueps$ and from an analytical perspective it can be defined as a suitable limit of $\ueps$ for $\eps \rightarrow 0$ (see Section \ref{subsection:homogenization:setting} below for more details). 

Since $\ueps$ is a continuous function in $t$, we restrict our considerations
to a fixed time $t$. Hence, we leave out the time dependency in the notation and denote e.g. $\ueps = \ueps \dott$. 

Let $\T_H$ denote a given coarse mesh and let $V_H \subset H^1_0(\Omega)$ denote a corresponding coarse finite dimensional subspace of $H^1_0(\Omega)$ that is sufficiently accurate to obtain good $L^2$-approximations. To quantify what we mean by \quotes{sufficiently accurate}, let $P_H$ denote the $L^2$-projection of $H^1_0(\Omega)$ on $V_H$, i.e. for $v \in H^1_0(\Omega)$ the projection $P_H(v)\in V_H$ fulfills
\begin{align}
\label{L2-projection}\int_{\Omega} P_H (v) w_H = \int_{\Omega} v w_H \qquad \mbox{for all } w_H \in V_H.
\end{align}
We assume that 
$$\| \ueps - P_H(\ueps) \|_{L^2(\Omega)} \le \delta_H ,$$
where $\delta_H$ is a given small tolerance. Let us denote $u_H :=P_{H}(\ueps) \in V_H$. 
Obviously, the $L^2$-projection will average out all small oscillations that cannot be seen on the coarse grid $\T_H$ (in this sense the projection homogenizes $\ueps$). By definition, $u_H$ is the best approximation of $\ueps$ in $V_H$ with respect to the $L^2$-norm. Next, we want to find a macroscopic equation that is fulfilled by $u_H$.

Since $\nabla u_H$ does not approximate $\nabla \ueps$, we are interested in a corrector $Q(u_H)$, such that
\begin{align*}
\int_{\Omega} \aeps (\nabla u_H + \nabla Q(u_H) ) \cdot \nabla v_H = \int_{\Omega} \aeps \nabla \ueps \cdot \nabla v_H  \qquad \mbox{for all } v_H \in V_H,
\end{align*}
or in a symmetric formulation
\begin{align}
\label{preliminary-method} \int_{\Omega} \aeps (\nabla u_H + \nabla Q(u_H) ) \cdot (\nabla v_H + \nabla Q(v_H) ) = \int_{\Omega} \aeps \nabla \ueps \cdot  (\nabla v_H + \nabla Q(v_H) )
\end{align}
for all $v_H \in V_H$. A suitable corrector operator $Q$ needs to fulfill two properties:\\
1. $Q(u_H)$ must be in the kernel of the $L^2$-projection $P_H$ in order to preserve the $L^2$-best-approximation property
\begin{align}
\label{trax-L2} \int_{\Omega} \ueps v_H =  \int_{\Omega} u_H \hspace{2pt} v_H = \int_{\Omega} (u_H + Q(u_H)) v_H \qquad \mbox{for all } v_H \in V_H.
\end{align}
2. It must incorporate the oscillations of $\aeps$. A natural way to achieve this is to make the ansatz
\begin{align*}
\int_{\Omega} \aeps (\nabla u_H + \nabla Q(u_H) ) \cdot \nabla v_h = 0,
\end{align*}
where the test function $v_h$ should be in $H^1_0(\Omega)$, but with the constraint $v_h \in \mbox{kern}(P_H)$. The constraint is sufficient to make the problem well posed (solution space and test function space are identical). 

In summary, we have the following strategy if $\ueps$ is a known function: find $u_H \in V_H$ that fulfills equation (\ref{preliminary-method}) and where for a given $v_H \in V_H$ the corrector $Q(v_H) \in \mbox{kern}(P_H)$ solves $\int_{\Omega} \aeps (\nabla v_H + \nabla Q(v_H) ) \cdot \nabla v_h = 0$ for all $v_h \in \mbox{kern}(P_H)$. Observe that this $u_H$ fulfills indeed $u_H = P_H(\ueps)$ as desired, because (by equation (\ref{trax-L2})) the function $e:=\ueps - u_H - Q(u_H)$ is in the kernel of $P_H$. Hence for all $v_H \in V_H$
\begin{align*}
\int_{\Omega} u_H v_H = \int_{\Omega} (u_H + Q(u_H)) v_H = \int_{\Omega} (u_H + Q(u_H) + e) v_H = \int_{\Omega} \ueps v_H,
\end{align*}
which means just $u_H=P_{H}(\ueps)$. Consequently, we also have the estimate
\begin{align}
\label{abstract-delta-est}\| \ueps - u_H \|_{L^2(\Omega)} \le \delta_H.
\end{align}
The only remaining problem is that we do not know the term $\int_{\Omega} \aeps \nabla \ueps \cdot (\nabla v_H + \nabla Q(v_H))$ on the right hand side of (\ref{preliminary-method}). However, we know that
\begin{align*}
\int_{\Omega} \aeps \nabla \ueps \nabla v = \int_{\Omega} F v - \int_{\Omega} \partial_{tt} \ueps \hspace{2pt} v.
\end{align*}
If the solution $\ueps$ is sufficiently regular then $\partial_{tt} \ueps$ is well approximated by $\partial_{tt} u_H =  P_{H}(\partial_{tt} \ueps)$ in the sense of (\ref{abstract-delta-est}).

This suggests to replace $\partial_{tt} \ueps$ by $\partial_{tt} u_H$ and to solve the approximate problem to find $\bar{u}_H \in V_H$ with
\begin{eqnarray*}
\int_{\Omega} \aeps (\nabla \bar{u}_H + \nabla Q(\bar{u}_H)) \cdot (\nabla v_H + \nabla Q(v_H)) &=&\int_{\Omega}(F-\partial_{tt} \bar{u}_H)
(\nabla v_H + \nabla Q(v_H))\\
&\approx& \int_{\Omega} \aeps \nabla \ueps (\nabla v_H + \nabla Q(v_H))
\end{eqnarray*}
for all $v_H \in V_H$.

Note that the above presented strategy is not yet a ready-to-use method, since the exact computation of the corrector $Q$ involves global fine scale problems. In order to overcome this difficulty, a localization of $Q$ is required together with a suitable fine scale discretization. The final method is presented in the Section \ref{section:multiscale:methods}. Before we can formulate the method, we introduce a suitable fully discrete space-discretization.

\section{The LOD method for the wave equation}
\label{main-section-LOD}

In this section we propose a space discretization for localizing the fine scale computations in the previously described ansatz. For that purpose, we make use of the tools of the Localized Orthogonal Decomposition (LOD) that was introduced in \cite{MaP14} (see also \cite{EGM13,HeM14,HeP13,HMP14,MaP15} for related works) and that originated from the Variational Multiscale Method (VMM, cf. \cite{HFM98,LaM07,Mal11}). We then derive our multiscale method for the wave equation with a continuum of scales.

\subsection{Spatial discretization}
\label{subsection-LOD}
The spatial discretization involves two discretization levels. On the one hand, we have a quasi-uniform coarse mesh on $\Omega$ that is denoted by $\T_H$. $\T_H$ consists either of conforming shape regular simplicial elements or of conforming shape regular quadrilateral elements. The elements are denoted by $K \in \T_H$ and the coarse mesh size $H$ is defined as the maximum diameter of an element of $\T_H$. On the other hand we have a fine mesh that is denoted by $\T_h$. It also consists of conforming and shape regular elements. Furthermore, we assume that $\T_h$ is obtained from an arbitrary refinement of $\T_H$, 
with the additional requirement that
$h \le (H/2)$, where $h$ denotes the maximum diameter of an element of $\T_h$. 
In practice we usually have $h \ll H$.
In particular, $\T_h$ needs to be fine enough to capture all the oscillations of $\aeps$. 
In contrast, the coarse mesh is only required to provide accurate $L^2$-approximations.

For $\T=\T_H,\T_h$ we denote
\begin{align}
\label{definition:finite:element:spaces}P_1(\T) &:= \{v \in C^0(\omega) \;\vert \;\forall K\in\T,v\vert_K \text{ is a polynomial of total degree}\leq 1\}, \\
\nonumber Q_1(\T) &:= \{v \in C^0(\omega) \;\vert \;\forall K\in\T,v\vert_K \text{ is a polynomial of partial degree}\leq 1\}.
\end{align}
With this, we define the classical coarse Lagrange finite element space $V_H$ by $V_H:=P_1(\T_h)\cap H^1_0(\Omega)$ for a simplicial mesh and by $V_H:=Q_1(\T_h)\cap H^1_0(\Omega)$ for a quadrilateral mesh. The fine scale space $V_h$ is defined in the same way. 

Subsequently, we will make use of the notation $a \lesssim b$ that abbreviates $a\leq Cb$, where $C$ is a constant that can dependent on $d$, $\Omega$, $\alpha$, $\beta$ and interior angles of the coarse mesh, but not on the mesh sizes $H$ and $h$. In particular it does not depend on the possibly rapid oscillations in $\aeps$. We write $a\lesssim_T b$ if $C$ is allowed to further depend on $T$ and the data functions $F$, $f$ and $g$.

The set of the interior Lagrange points (interior vertices) of the coarse grid $\T_H$ is denoted by $\mathcal{N}_H$. For each node $z \in \mathcal{N}_H$ we let $\Phi_z \in V_H$ denote the corresponding nodal basis function that fulfills $\Phi_z(z)=1$ and $\Phi_z(y)=y$ for all $y \in \mathcal{N}_H \setminus \{ z \}$.

In the next step, we define the kernel of the $L^2$-projection (\ref{L2-projection}) restricted to $V_h$ in a slightly alternative way. Recall that this kernel was required as the solution space for the corrector problems discussed in Section \ref{section:motivation}. However, from the computational point of view it is more suitable to not work with the $L^2$-projection directly, since it involves to solve a system of equations in order to verify if an element is in the kernel. For that reason, we subsequently express kern$(P_H\vert_{V_h})$ equivalently by means of a weighted Cl\'ement-type quasi-interpolation operator $I_H$ (cf. \cite{Car99}) that is defined by
\begin{align}
\label{def-weighted-clement} I_H : H^1_0(\Omega) \rightarrow V_H,\quad v\mapsto I_H(v):= 
\sum_{z \in \mathcal{N}_{H}}
v_z \Phi_z \quad \text{with }v_z := \frac{(v,\Phi_z)_{L^2(\Omega)}}{(1,\Phi_z)_{L^2(\Omega)}}.
\end{align}
With that, we define 
$$W_h := \mbox{\rm {Ker}}(I_H\vert_{V_h}).$$
Indeed the space $W_h$ is the previously discussed kernel of the $L^2$-projection. This claim can be easily verified: if $I_H(v_h)=0$ for an element $v_h \in V_h$, then we have by the definition of $I_H$ that $(v_h,\Phi_z)_{L^2(\Omega)}=0$ for all $z \in \mathcal{N}_H$. Since $\Phi_z$ is just the nodal basis of $V_H$ we have $(v_h,\Phi_H)_{L^2(\Omega)}=0$ for any $\Phi_H \in V_H$. Hence $W_h=\mbox{Ker}(I_H\vert_{V_h}) \subset \mbox{Ker}(P_H\vert_{V_h})$. The reverse inclusion $\mbox{Ker}(P_H\vert_{V_h}) \subset W_h$ is straightforward. In particular, we have the splitting $V_h = \mbox{Im}(P_H\vert_{V_h}) \oplus \mbox{Ker}(P_H\vert_{V_h}) = {\mbox{Im}(I_H\vert_{V_h}) \oplus \mbox{Ker}(I_H\vert_{V_h})}= V_H \oplus W_h$.

We can now define the optimal corrector $Q_{h,{\Omega}}:V_H \rightarrow W_h$ (in the sense of Section \ref{section:motivation}) as the solution $Q_{h,{\Omega}}(v_H) \in W_h$ of
\begin{align}
\label{projection-orthogonal}
\beps( v_H + Q_{h,{\Omega}}(v_H) , w_h ) = 0 \qquad \mbox{for all } w_h \in W_h,
\end{align}
where $\beps(\cdot,\cdot)$ is defined in (\ref{definition:problem:bilinearform}). However, finding $Q_{h,{\Omega}}(v_H)$ involves a problem in the whole fine space $V_h$ and is therefore very expensive. For this purpose, we wish to localize the corrector $Q_{h,{\Omega}}$ by element patches.

For $k\in \mathbb{N}$, we define patches $U_k(K)$ that consist of a coarse element $K \in \T_H$ and $k$-layers of coarse elements around it. More precisely $U_k(K)$ is defined iteratively by
\begin{equation}\label{def-patch-U-k}
    \begin{aligned}
      U_0(K) & := K, \\
      U_k(K) & := \cup\{T\in \T_H\;\vert\; T\cap U_{k-1}(K)\neq\emptyset\}\quad k=1,2,\ldots .
    \end{aligned}
\end{equation}
Practically, we will later see that we only require small values of $k$ (typically $k=1,2,3$).

With that, we define the localized corrector operator in the following way:
\begin{definition}[Localized Correctors]
\label{definition:localized:ms:space}
For $k\in \mathbb{N}$, $K\in \T_H$ and $U_k(K)$ defined according to (\ref{def-patch-U-k}), we define the localized version of $\mbox{Ker}(P_H\vert_{V_h})$ by
$$W_h(U_k(K)):=\{ w_h \in W_h| \hspace{2pt} w_h=0 \enspace \mbox{in } \Omega \setminus U_k(K) \}.$$
The localized version of the operator (\ref{projection-orthogonal}) can be constructed in the following way. First, for $v_H \in V_H$ find $Q_{h,k}^{K}(v_H) \in {W}_h(U_k(K))$ with
\begin{align}
\label{local-corrector-problem}\int_{U_k(K)} \aeps \nabla Q_{h,k}^{K}(v_H)\cdot \nabla w_h = - \int_K \aeps \nabla v_H \cdot \nabla w_h \qquad \mbox{for all } w_h \in W_h(U_k(K)).
\end{align}
Then, the global approximation of $Q_{h,k}$ is defined by
\begin{align}
\label{global-corrector}Q_{h,k}(v_H):=\sum_{K\in \T_H} Q_{h,k}^{K}(v_H).
\end{align}
Observe that if $k$ is large enough so that $U_k(K)=\Omega$ for all $K \in \T_H$ (a case that is only useful for the analysis), we have $Q_{h,k}=Q_{h,{\Omega}}$, where $Q_{h,{\Omega}}$ is the corrector operator introduced in (\ref{projection-orthogonal}).
\end{definition}

\begin{remark}[Splittings of $V_h$]
The space that is spanned by the image of $(I+Q_{h,k})_{\vert V_H}$ is given by
\begin{align}
\label{localized-ms-space}
\Vmsk:=\{ v_H + Q_{h,k}(v_H)| \hspace{2pt} v_H \in V_H\}.
\end{align}
Furthermore, we denote $\Vms:=\{ v_H + Q_{h,\Omega}(v_H)| \hspace{2pt} v_H \in V_H\}$ for the optimal corrector. This gives us the following splittings of $V_h$:
\begin{align*}
V_h &= V_H \oplus W_h, \qquad \hspace{7pt} \mbox{where } \quad V_H  \perp W_h \quad \hspace{7pt} \mbox{w.r.t. } (\cdot , \cdot)_{L^2(\Omega)},\\
V_h &= \Vms \oplus W_h, \qquad \mbox{where } \quad \Vms  \perp W_h \quad \mbox{w.r.t. } \beps (\cdot , \cdot),\\
V_h &= \Vmsk \oplus W_h.
\end{align*}
\end{remark}

Beside the operator $I_H$ that we defined in (\ref{def-weighted-clement}) other choices of interpolation operators (such as the classical Cl\'ement interpolation) are possible to construct splittings $V_h = \Vmsk \oplus W_h$ with $\Vmsk \perp_{\beps (\cdot , \cdot)} W_h$. If the operator fulfills various standard properties (like interpolation error estimates, $H^1$-stability, etc.; cf. \cite{HMP14} for an axiomatic list) the space $\Vmsk$ will have similar approximation properties as the multiscale space that we use in this contribution. However, we note that the particular Cl\'ement-type interpolation operator from (\ref{def-weighted-clement}) yields the $L^2$-orthogonality $V_H  \perp W_h$ which is typically not the case for other operators. This is central in our approach. In this paper we particularly exploit this feature to show that we obtain higher order convergence rates under the assumption of additional regularity. We also note that the Lagrange-interpolation fails to yield good approximations (cf. \cite{HeP13}).

\begin{remark}
Observe that the solutions $Q_{h,k}(v_H)$ of (\ref{local-corrector-problem}) are well defined by the Lax-Milgram theorem. Furthermore, it is was shown that solutions such as $Q_{h,k}(v_H)$ (with localized source term) decay with exponential speed to zero outside of the support of the source term (cf. \cite{MaP14,HeM14}). More precisely, we will later see that we have an estimate of the type $\| \nabla (Q_{h,k} - Q_{h,\Omega})(v_H) \|_{L^2(\Omega)} \lesssim k^{d/2} \theta^{k} \| \nabla v_H \|_{L^2(\Omega)}$ for a generic constant $0<\theta <1$. Hence, we have exponential convergence in $k$ and small values for $k$ (typically $k=2,3$) can be used to get accurate approximations of $Q_{h,\Omega}$. For small values of $k$, the local problems (\ref{local-corrector-problem}) are affordable to solve, they can be solved in parallel and $Q_{h,k}(\Phi_z)$ is only locally supported for every nodal basis function $\Phi_z \in V_H$.
\end{remark}

\subsection{Semi and fully-discrete multiscale method}
\label{section:multiscale:methods}

Based on the discretization and the correctors defined in Section \ref{subsection-LOD}, we present a semi-discrete multiscale method for the wave equation and state a corresponding a priori error estimate. As an example of a time-discretization, we also present a Crank-Nicolson realization of the method and state a fully discrete space-time error estimate. In order to abbreviate the notation, we define effective/macroscopic bilinear forms for $v_H,w_H \in V_H$ by
\begin{align*}
\bHk (v_H, w_H) &:= \beps( v_H + Q_{h,k}(v_H), w_H + Q_{h,k}(w_H)) \quad \mbox{and}\\
 (v_H, w_H)_{H,k} &:= ( v_H + Q_{h,k}(v_H), w_H + Q_{h,k}(w_H))_{L^2(\Omega)},
\end{align*}
where $\beps$ is defined in (\ref{definition:problem:bilinearform}).

\subsubsection{Semi-discrete multiscale method}

We can now formulate the method.
\begin{definition}[Semi-discrete multiscale method]
Let $k \in \mathbb{N}$ denote the localization parameter that determines the patch size $U_k(K)$ for $K\in \T_H$ (according to (\ref{def-patch-U-k})) and hence also determines the localized corrector operator $Q_{h,k}$. The semi-discrete approximation $\uHk \in H^2(0,T;V_H)$ (of $\ueps$ in $L^2$) solves the following system for all $v_H \in V_H$  and $t > 0$
\begin{align}
\label{semi-discrete-lod-equation}
\nonumber ( \partial_{tt} \uHk \dott, v_H )_{H,k} + \bHk( \uHk \dott,  v_H ) &= \left( F \dott, v_H + Q_{h,k}(v_H)\right)_{L^2(\Omega)},\\
(\uHk+Q_{h,k}(\uHk))( \cdot , 0 ) &= \pi_{H,k}^{\ms}(f),\\
\nonumber \partial_t (\uHk+Q_{h,k}(\uHk))( \cdot , 0 ) &= P_{H,k}^{\ms}(g),
\end{align}
with the projections $\pi_{H,k}^{\ms}$ and $P_{H,k}^{\ms}$ defined in (\ref{elliptic-projection-pi_Hk}) and (\ref{L2-projection-P_Hk}) below.
\end{definition}
Recall $\Vmsk$ be the space defined in (\ref{localized-ms-space}). We subsequently define two elliptic projections for $v \in L^{2}(0,T;H^1_0(\Omega))$ and one $L^2$-projection.
\begin{enumerate}
\item The projection $\pi_h : L^{2}(0,T;H^1_0(\Omega)) \rightarrow L^{2}(0,T;V_h)$ is given by:\\
find $\pi_h(v) \in L^{2}(0,T;V_h)$ with
\begin{align}
\label{elliptic-projection-pi_h}\beps ( \pi_{h}(v)(\cdot,t) , w ) = \beps ( v(\cdot,t) , w ) \qquad \mbox{for all } w \in V_h, \quad \mbox{for almost every } t \in (0,T).
\end{align}
\item The projection $\pi_{H,k}^{\ms} : L^{2}(0,T;H^1_0(\Omega)) \rightarrow L^{2}(0,T;\Vmsk)$ is given by:\\
find $\pi_{H,k}^{\ms}(v)(\cdot,t) \in \Vmsk$ with
\begin{align}
\label{elliptic-projection-pi_Hk}\beps( \pi_{H,k}^{\ms}(v)(\cdot,t) , w ) = \beps ( v(\cdot,t) , w ) \qquad \mbox{for all } w \in \Vmsk, \quad \mbox{for almost every } t \in (0,T).
\end{align}
For $U_k(K)=\Omega$, we denote by $\pi_{H,\Omega}^{\ms}$ the above projection mapping from $L^{2}(0,T;H^1_0(\Omega))$ to $L^{2}(0,T;\Vms)$.
\item The $L^2$-projection $P_{H,k}^{\ms} : L^{2}(0,T;H^1_0(\Omega)) \rightarrow L^{2}(0,T;\Vmsk)$ is given by:\\

find $P_{H,k}^{\ms}(v)(\cdot,t) \in \Vmsk$ with
\begin{align}
\label{L2-projection-P_Hk}( P_{H,k}^{\ms}(v)(\cdot,t) , w )_{L^2(\Omega)} = ( v(\cdot,t) , w )_{L^2(\Omega)} \qquad \mbox{for all } w \in \Vmsk, \quad \mbox{for a.e. } t \in (0,T).
\end{align}
\end{enumerate}

\begin{remark}[Existence and uniqueness]
If assumption (H0) is fulfilled, the system (\ref{semi-discrete-lod-equation}) has a unique solution. This result directly follows from standard ODE theory, after a reformulation of (\ref{semi-discrete-lod-equation}) into a (finite) system of first order ODE's with constant coefficients and applying Duhamel's formula. Due to the Sobolev embedding theorems, we have $\uHk \in C^1([0,T];V_H)$. If additionally $f \in C^0(0,T;L^2(\Omega))$, we even get $\uHk \in C^2([0,T];V_H)$. The corrector $Q_{h,k}(\uHk)$ inherits this time regularity.
\end{remark}

\subsubsection{Fully discrete multiscale method}

In this section, for $J\in \mathbb{N}$, we let $\triangle t:=\frac{T}{J}>0$ denote the time step size and we define $t^n:=n \triangle t$ for $n \in \mathbb{N}$. 
In order to propose a time
discretization of (\ref{semi-discrete-lod-equation}), we introduce some simplifying notation. First, recall that for every coarse interior node $z \in \mathcal{N}_H$, we denote the corresponding nodal basis function by $\Phi_z$. The total number of these coarse nodes shall be denoted by $N_H$ and we assume that they are ordered by some index set, i.e. $\mathcal{N}_H= \{ z_1, \dots, z_N \}$. With that we define the corresponding stiffness matrix $S_{k} \in \mathbb{R}^{N \times N}$ by the entries
\begin{align*}
(S_{k})_{ij} := \bHk( \Phi_{z_j},  \Phi_{z_i} )
\end{align*}
and the entries of the corrected mass matrix $M_{k} \in \mathbb{R}^{N \times N}$ by 
\begin{align*}
(M_{k})_{ij} := ( \Phi_{z_j},  \Phi_{z_i} )_{H,k}.
\end{align*}
The load vectors $G_{k},\bar{f}_{k},\bar{g}_{k} \in \mathbb{R}^{N}$ arising in (\ref{semi-discrete-lod-equation}) are defined by
\begin{align}
\label{def-data-vectors-cn}\nonumber(G_{k})_i(t) &:= \left( F \dott, \Phi_{z_i} + Q_{h,k}(\Phi_{z_i}) \right)_{L^2(\Omega)},\\
(\bar{f}_{k})_i \mbox{  is such that } \pi_{H,k}^{\ms}(f) &= \sum_{i=1}^N (\bar{f}_{k})_i (\Phi_{z_i} + Q_{h,k}(\Phi_{z_i})),\\
(\bar{g}_{k})_i \mbox{  is such that } P_{H,k}^{\ms}(g) &= \sum_{i=1}^N (\bar{g}_{k})_i (\Phi_{z_i} + Q_{h,k}(\Phi_{z_i})),
\end{align}
with $\pi_{H,k}^{\ms}$ being the elliptic projection on $\Vmsk$ (see (\ref{elliptic-projection-pi_Hk})) and
$P_{H,k}^{\ms}$ being the $L^2$-projection on $\Vmsk$ (see (\ref{elliptic-projection-pi_Hk})).
Hence, we can write (\ref{semi-discrete-lod-equation}) as the system: find $\xi_k(t)\in \mathbb{R}^{N}$ with 
\begin{align*}
M_k \ddt\xi_k(t) + S_k \xi_k(t) &= G_k(t), \quad \hspace{10pt} \mbox{for } 0<t<T
\end{align*}
and $\xi_k(0)=\bar{f}_{k}$ and $\dt \xi_k(0)=\bar{g}_{k}$. This yields $\uHk \dott = \sum_{i=1}^N (\xi_k(t))_i \Phi_{z_i}$.

In order to solve this system we can apply the Newmark scheme.

\begin{definition}[Newmark scheme]
For $n\ge 1$, given initial values $\xi_k^{(0)} \in \R^N$ and $\xi_k^{(1)} \in \R^N$, and given load vectors $G_k^{(n)} \in \R^N$, we 
define the Newmark approximation $\xi_k^{(n+1)}$ of $\xi_k(t^{(n+1)})$ iteratively as the solution $\xi_k^{(n+1)} \in \mathbb{R}^{N}$ of
\begin{eqnarray*}
\lefteqn{(\triangle t)^{-2} M_k \left( \xi_k^{(n+1)} -2 \xi_k^{(n)} + \xi_k^{(n-1)} \right)}\\
&+& \frac{1}{2} S_k \left( 2 \hat{\beta} \xi_k^{(n+1)} + (1 - 4 \hat{\beta} + 2 \hat{\gamma} ) \xi_k^{(n)} + (1 + 2\hat{\beta} -2\hat{\gamma}) \xi_k^{(n-1)} \right) = G_k^{(n)}.
\end{eqnarray*}
Here, $\hat{\beta}$ and $\hat{\gamma}$ are given parameters.

An example for an implicit method is given by the choice $\hat{\beta}=1/4$ and $\hat{\gamma}=1/2$, which leads to the classical Crank-Nicolson scheme. Another example is the leap-frog scheme that is obtained for $\hat{\beta}=0$ and $\hat{\gamma}=1/2$. The leap-frog scheme is explicit (up to a diagonal mass matrix which can be obtained by mass lumping).
\end{definition}

As one possible realization, we subsequently consider the case $\hat{\beta}=1/4$ and $\hat{\gamma}=1/2$, i.e. the Crank-Nicolson scheme (see Definition \ref{definition-fully-discrete-method} below). We state a corresponding a priori error estimate and the numerical experiments in Section \ref{section-numerics} are also performed with this method. Before we present the main theorems, let us detail the method by specifying the initial values and the load vectors for the Crank-Nicolson method.

\begin{definition}[Fully-discrete Crank-Nicolson multiscale method]
\label{definition-fully-discrete-method}
As before, let $k \in \mathbb{N}$ denote the localization parameter that determines the patch size $U_k(K)$ for $K\in \T_H$ (according to (\ref{def-patch-U-k})). The load functions $G_{k},\bar{f}_{k},\bar{g}_{k} \in \mathbb{R}^{N}$ are defined according to (\ref{def-data-vectors-cn}) and we denote 
$G_k^{(n)}:=\frac{1}{2}( G_k{t^n} + G_k(t^{n-1}) )$ for $t^n=n \triangle t$, $n\ge 1$. Defining $\xi_k^{(0)}:=\bar{f}_{k}$ and $\eta_k^{(0)}:=\bar{g}_{k}$, the approximation $(\xi_k^{(n)},\eta_k^{(n)}) \in \R^N \times \R^N$ in the $n$'th time step is given as the solution of the linear system
\begin{align*}
(\frac{\triangle t^2}{4} S_k + M_k) \eta_k^{(n)} &= (M_k - \frac{\triangle t^2}{4} S_k ) \eta_k^{(n-1)} - \triangle t S_k \xi_k^{(n-1)} + \triangle t G_k^{(n)}
\end{align*}
and $\xi_k^{(n)}:= \frac{\triangle t}{2} \eta_k^{(n)} +\frac{\triangle t}{2} \eta_k^{(n-1)} + \xi_k^{(n-1)}$.

With that, the Crank-Nicolson approximation of (\ref{semi-discrete-lod-equation}) is defined as the piecewise linear function $\uHtk$ with 
\begin{align}
\label{fully-discrete-lod-approximation}\uHtk(\cdot,t) := \sum_{i=1}^N \left( \frac{t^{n+1}-t}{\triangle t} (\xi_k^{(n)})_i + \frac{t-t^{n}}{\triangle t} (\xi_k^{(n+1)})_i \right)\Phi_{z_i} \qquad \mbox{for } t \in [t^{n},t^{n+1}].
\end{align}
\end{definition}

\begin{remark}
Existence and uniqueness of $(\xi_k^{(n)},\eta_k^{(n)})$ in Definition \ref{definition-fully-discrete-method} is obvious
since the system matrix $(\frac{\triangle t^2}{4} S_k + M_k)$ has only positive eigenvalues.
\end{remark}

\section{Main results and survey of other multiscale strategies}
\label{section-main-results-and-misc}

\subsection{Convergence results under minimal assumptions on the initial data}
\label{subsection:homogenization:setting}
In this section we recall some fundamental results concerning the homogenization of the wave equation and relate it to our multiscale method. 
In this framework, we are able to give a convergence result between the  homogenized solution and our multiscale approximation with the weakest possible assumptions: no scale separations, no assumptions on the initial data except the one that guaranties existence and uniqueness of the original problem \eqref{wave-equation-strong}.

The essential question of classical homogenization is the following: if $(\aeps)_{\eps>0}$ represents a sequence of coefficients and if we consider the corresponding sequence of solutions $(\ueps)_{\eps>0}$ of (\ref{wave-equation-weak}), does $\ueps$ converge in some sense to a $\uhom$ that we can characterize in a simple way? The hope is that $\uhom$ fulfills some equation (the homogenized equation) which is cheap to solve since it does no longer involve multiscale features (which were averaged out in the limit process for $\eps \rightarrow 0$). With the abstract tool of $G$-convergence it is possible to answer this question:

\begin{definition}[$G$-convergence]\label{def-G-convergence}
A sequence $(\aeps)_{\eps >0}\subset \mathcal{M}(\alpha,\beta,\Omega)$ (i.e. with uniform spectral bounds in $\eps$) is said to be $G$-convergent to $\ahom \in \mathcal{M}(\alpha,\beta,\Omega)$ if for all $F \in H^{-1}(\Omega)$ the sequence of solutions $v^{\eps} \in H^1_0(\Omega)$ of
\begin{align*}
\int_{\Omega} \aeps \nabla v^{\eps} \cdot \nabla v = F(v) \qquad \mbox{for all } v \in H^1_0(\Omega)
\end{align*}
satisfies $v^{\eps} \rightharpoonup v^0$ weakly in $H^1_0(\Omega)$, where $v^0 \in H^1_0(\Omega)$ solves
\begin{align*}
\int_{\Omega} \ahom \nabla v^0 \cdot \nabla v = F(v) \qquad \mbox{for all } v \in H^1_0(\Omega).
\end{align*}
\end{definition}
One of the main properties of $G-$convergence is the following compactness result \cite{Spa68,Tar77}: let $(\aeps)_{\eps >0}$ be a sequence of matrices in  $\mathcal{M}(\alpha,\beta,\Omega)$, then there exists a subsequence $(a^{\eps '})_{\eps '>0}$ and a matrix $a^0\in \mathcal{M}(\alpha,\beta,\Omega)$ such that $(a^{\eps '})_{\eps '>0}$ $G-$converges to $a^0$. For the wave equation, we have the following result obtained in \cite[Theorem 3.2]{BFM92}:

\begin{theorem}[Homogenization of the wave equation]
\label{theorem:homogenization:wave:equation} Let assumptions (H0) be fulfilled and let the sequence of symmetric matrices $(\aeps)_{\eps >0}\subset \mathcal{M}(\alpha,\beta,\Omega)$ be $G$-convergent to some $\ahom \in \mathcal{M}(\alpha,\beta,\Omega)$. Let $\ueps \in L^{\infty}(0,T;H^1_0(\Omega))$ denote the solution of the wave equation (\ref{wave-equation-weak}). Then it holds
\begin{align*}
\ueps &\rightharpoonup \uhom \quad \mbox{weak-}\ast \mbox{ in } L^{\infty}(0,T,H^1_0(\Omega)),\\
\partial_t \ueps &\rightharpoonup \partial_t \uhom \quad \mbox{weak-}\ast \mbox{ in } L^{\infty}(0,T,L^2(\Omega))
\end{align*}
and where $\uhom \in L^2(0,T;H^1_0(\Omega))$ is the unique weak solution of the homogenized problem
\begin{align}
\label{homogenized-equation-weak}
\nonumber\langle \partial_{tt} \uhom \dott, v \rangle + \left( \ahom \nabla \uhom \dott, \nabla v \right)_{L^2(\Omega)} &= \left( F \dott, v \right)_{L^2(\Omega)} \qquad \mbox{for all } v \in H^1_0(\Omega) \mbox{ and } t > 0,\\
\left( \uhom( \cdot , 0 ) , v \right)_{L^2(\Omega)} &= \left( f , v \right)_{L^2(\Omega)} 
\qquad \hspace{21pt} \mbox{for all } v \in H^1_0(\Omega),\\
\nonumber \left( \partial_t \uhom( \cdot , 0 ) , v \right)_{L^2(\Omega)} &= \left( g , v \right)_{L^2(\Omega)}
\quad \hspace{32pt} \mbox{for all } v \in H^1_0(\Omega).
\end{align}
\end{theorem}
This Theorem and the compactness result stated above show that for {\it any} problem \eqref{wave-equation-strong} based on a sequence of matrices with $\aeps\in\mathcal{M}(\alpha,\beta,\Omega)$, we can extract a subsequence such that the corresponding solution of the wave problem convergence to a homogenized solution. Except for special situations, e.g., locally periodic coefficients $\aeps$, i.e. tensor $\aeps(x)=a(x,\frac{x}{\eps})$ that are $\eps$-periodic on a fine scale or for random stationary tensors, it is not possible to construct $\ahom$ explicitly. The next theorem states that the solution 
$u_{H,K}$ of our multiscale method  \eqref{semi-discrete-lod-equation} converges to the homogenized solution also in the case where no explicit solution of 
$\ahom$ is known. This is a convergence result for the multiscale method to the effective solution of problem \eqref{wave-equation-strong}
under the weakest possible assumptions. Let $\ahom$ denote the $G$-limit of $\aeps$ and let $f^{\eps} \in H^1_0(\Omega)$ be the solution to
\begin{align}\label{def-f-eps}
\int_{\Omega} \aeps \nabla f^{\eps} \cdot \nabla v = \int_{\Omega} \ahom \nabla f \cdot \nabla v \qquad \mbox{for all } v \in H^1_0(\Omega),
\end{align}
where $f$ is the initial value in problem \eqref{wave-equation-strong}.
By the definition of $G$-convergence, we have that $f^{\eps} \rightharpoonup f$ weakly in $H^1_0(\Omega)$ and hence $f^{\eps} \rightarrow f$ strongly in $L^2(\Omega)$. Define  $e_{\mbox{\tiny \rm hom}}=\| \ueps - u^0 \|_{L^{\infty}(L^2)} +\| f^\eps - f \|_{(L^2)}$.
Note that under the assumptions of Theorem \ref{theorem:homogenization:wave:equation}, we have $\lim_{\eps \rightarrow 0} e_{\mbox{\tiny \rm hom}}(\eps) = 0.$

\begin{theorem}[A priori error estimate for the homogenized solution]\label{apriori-lod-homogenization}$\\$
Consider the setting of Theorem \ref{theorem:homogenization:wave:equation}, let $h<\eps <H$ and assume (H0), $g\in H^1_0(\Omega)$, $\partial_t F \in L^2(0,T,L^2(\Omega))$ and $\nabla \cdot (\ahom \nabla f) + F(\cdot,0)\in L^2(\Omega)$.
By $\uhom$ we denote the homogenized solution given by (\ref{homogenized-equation-weak}) and by $\uHk$ we denote the solution of 
\eqref{semi-discrete-lod-equation}.
Under these assumptions there exists a generic constant $C_\theta$ (i.e. independent of $H$, $h$ and $\eps$) such that if $k \ge C_{\theta} |\ln(H)|$ 
it holds
\begin{eqnarray}
\label{main-result-est-1}
\lim_{\eps\rightarrow 0} \lim_{h\rightarrow 0}
\| \uhom - \uHk \|_{L^{\infty}(L^2)} &\lesssim_T& H \\
\label{main-result-est-2}
\lim_{h\rightarrow 0}\| \ueps - \uHk \|_{L^{\infty}(L^2)} &\lesssim_T&  H + e_{hom}(\eps).
\end{eqnarray}
If we replace the elliptic projection $\pi_{H,k}^{\ms}(f)$ in (\ref{semi-discrete-lod-equation}) by the $L^2$-projection $P_{H,k}^{\ms}(f)$, 
the estimate
still remains valid.
\end{theorem}
The theorem is proved in Section \ref{section:proofs}, where the dependencies on $\eps$ and $h$ are elaborated. 
In particular, the following sharpened result can be extracted from the proof of Theorem \ref{apriori-lod-homogenization}:
Let $g\in H^1_0(\Omega)$ and $\partial_t F \in L^2(0,T,L^2(\Omega))$. If $f=0$ the estimate in Theorem \ref{apriori-lod-homogenization} can be improved to
\begin{align}
\label{main-result-est-1-sharp}
\lim_{h\rightarrow 0}\| \ueps - \uHk \|_{L^{\infty}(L^2)} \lesssim_T H.
\end{align}
Estimate \eqref{main-result-est-1} 
guarantees convergence of our method 
with respect to the $L^{\infty}(L^2)$-error under the weakest possible assumptions in the general setting of $G$-convergence without any restrictions on the initial values. However, for some choices of the initial values, these estimates can be still improved significantly. This case is discussed in the next section.

\subsection{Convergence results for well-prepared initial data}
\label{subsection-well-prepared-initial-values}

In the previous section, we showed convergence of our method in the setting of $G$-convergence. We obtained a linear rate in $H$ for the $L^{\infty}(L^2)$-error. In this section, we show that this convergence can be improved for {\it well-prepared} initial values by using correctors from the kernel of the $L^2$-projection. In particular, we obtain $L^{\infty}(H^1)$- and $W^{1,\infty}(L^2)$-error estimates with respect to the exact solution $\ueps$. In a first step, we define what we mean by {\it well-prepared initial values} and why they are crucial for improved estimates. Consequently, we need to assume that $f$, $g$ and $F(\cdot,0)$ are $\eps$-dependent such that they interact constructively with $\aeps$ (in the sense specified in Proposition \ref{proposition-time-regularity} below). We note that in this section, $\eps$  is an abstract parameter and functions $z$ with superscript $\eps$  
are assumed to have a large  $\| z^\eps \|_{H^s(\Omega)}$ norm for $s>1$.
For the wave equation, this kind of blow-up cannot only be triggered by the spatial derivatives, but also by the time derivatives. 
Typically we have a large $\| \ueps \|_{W^{m,2}(0,T;H^s(\Omega))}$ norm when $m+s>1$ (in a homogenization context 
$\| \ueps \|_{W^{m,2}(0,T;H^s(\Omega))} \rightarrow \infty$ for $\eps \rightarrow 0$ when $m+s>1$, see e.g., \cite{CiD99}). However, this statement can be relaxed. More precisely, under certain assumptions on the initial data, it is possible to show that $\| \ueps \|_{W^{m,2}(0,T;H^1(\Omega))}$ remains bounded independent of $\eps$. To make this statement precise, we state the following regularity result.

\begin{proposition}[Time-regularity and regularity estimates]\label{proposition-time-regularity}
Let assumption (H0) be fulfilled and let $F \in W^{m,2}(0,T;L^2(\Omega))$ for some $m\in \mathbb{N}$. Furthermore we define iteratively 
\begin{align}
\label{higher-order-initial-values}w_0^{\eps}:=f, \qquad w_1^{\eps}:=g, \qquad w_j^{\eps} := \partial_t^{j-2} F(\cdot , 0) + \nabla \cdot ( \aeps \nabla w_{j-2}^{\eps}) \quad \mbox{for } j=2,3,\cdots,m+1. 
\end{align}
If $w_j^{\eps} \in H^1_0(\Omega)$ for $0\le j \le m$ and $w_{m+1}^{\eps}\in L^2(\Omega)$, we have
$$
\partial_t^{m} \ueps \in L^{\infty}(0,T;H^1_0(\Omega)); \enspace  \partial_t^{m+1} \ueps  \in L^{\infty}(0,T;L^2(\Omega))
\enspace \mbox{and} \enspace  \partial_t^{m+2} \ueps \in L^{2}(0,T;H^{-1}(\Omega))
$$
and the regularity estimate
\begin{eqnarray}
\label{reg-estimates}
\nonumber\lefteqn{\| \partial_t^{m} \ueps \|_{L^{\infty}(0,T;H^1(\Omega))} + \| \partial_t^{m+1} \ueps \|_{L^{\infty}(0,T;L^2(\Omega))}}\\
&\lesssim_T& \| F \|_{W^{m,2}(0,T;L^2(\Omega))} + \sum_{j=0}^m \| w_j^{\eps} \|_{H^1(\Omega)} +  \| w_{m+1}^{\eps} \|_{L^2(\Omega)}.
\end{eqnarray}
\end{proposition}

A proof of this proposition can be extracted from the results presented in \cite[Chapter 7.2]{Eva10}. From the regularity estimate \eqref{reg-estimates} we see that the higher order time derivatives of $\ueps$ might not be independent of $\eps$ (i.e. the data oscillations) if the initial values $f$, $g$ and $\partial_t^{m} F(\cdot,0)$ are not {\it well-prepared}, i.e. if the functions $w_j^{\eps}$ cannot be bounded independent of $\eps$. In the previous section, we learned that the method shows a convergence of linear order for the $L^{\infty}(L^2)$-error, even if $\| w_2^{\eps}\|_{L^2(\Omega)}$ grows with $\eps$ (not well-prepared). This should be seen as a worst-case result. In this section, we want to derive improved estimates for the case of {\it well-prepared} initial values. In the light of the regularity estimate, we hence formulate the following compatibility condition.

\begin{definition}[Compatibility condition and well-preparedness]
\label{def-of-comp-condition}
Let $k\in \mathbb{N}$ and let $w_0^{\eps},\ldots,w_{k+1}^{\eps}$ be the initial data defined in \eqref{higher-order-initial-values}. We say that the data is {\it well-prepared and compatible of order $k$} if $F \in W^{k,2}(0,T;L^2(\Omega))$, $w_j^{\eps} \in H^1_0(\Omega)$ for $0\le j \le k$ and $w_{k+1}^{\eps}\in L^2(\Omega)$ and if 
\begin{align}
\label{bound-well-prep}\sum_{j=0}^k \| w_j^{\eps} \|_{H^1(\Omega)} +  \| w_{k+1}^{\eps} \|_{L^2(\Omega)} \le C_w,
\end{align}
where $C_w$ denotes a (generic) constant that can depend on $F,f,g,\Omega,\alpha$ and $\beta$, but not on $\eps$.
\end{definition}

\begin{remark}[Fulfillment of compatibility and well-preparedness]
Observe that the initial data is always 
well-prepared and compatible of order $0$. Furthermore, if $f=0$, $g=0$ and $\partial_t^{j} F(\cdot , 0)=0$ for $0\le j \le k-2$, the compatibility condition of order $k$ is trivially fulfilled. For any other case, we note that $C_w$ in \eqref{bound-well-prep} is a computable constant, since $w_j^{\eps}$ are known functions. Consequently, we can check a priori if the initial value is compatible and well-prepared.
\end{remark}
Well-prepared initial data is often crucial in homogenization settings. For instance, for deriving a homogenized model that captures long-time dispersion in the multiscale wave equation, well-prepared initial values are essential \cite{Lam11b}. Similar observations can be made for parabolic problems with a large drift (cf. \cite{AlO07}) which also can be seen as having a hyperbolic character.

The next main result of this paper are high order convergence rates in $L^{\infty}(L^2)$, provided that we use the correctors $Q_{h,k}( \uHk )$ and that the data is {\it well-prepared and compatible}. Furthermore, we also show that the method yields convergence in $W^{1,\infty}(L^2)$ and $L^{\infty}(H^1)$.
For arbitrary initial values, we can only guarantee a convergence in $L^{\infty}(L^2)$ according to Theorem \ref{apriori-lod-homogenization}. In the following theorem, $\ueps$ denotes the exact solution of the wave equation (\ref{wave-equation-weak}) and $\uHk$ the numerically homogenized solution defined by (\ref{semi-discrete-lod-equation}).
\begin{theorem}[A priori corrector error estimates for the semi-discrete method]\label{apriori-semidiscrete}$\\$
Assume (H0) and that the data is {\rm well-prepared and compatible of order $2$} in the sense of Definition \ref{def-of-comp-condition}. Then 
there exists a generic constant $C_\theta$ (i.e. independent of $H$, $h$ and $\eps$) such that if
$k \ge C_{\theta} |\ln(H)|$
the following a priori error estimates hold:
\begin{align}
\label{main-result-est-2a}
\|  \ueps - (\uHk + Q_{h,k}(\uHk)) \|_{L^{\infty}(L^2)} \lesssim_T H^2 + e_{\mbox{\tiny \rm disc}}^{\mbox{\tiny$(1)$}}(h);
\end{align}
if $F \in L^{\infty}(H^1_0)\cap W^{1,1}(H^1_0)$ and if the data is {\rm well-prepared and compatible of order $3$}:
\begin{align}
\label{main-result-est-3a}
\| \ueps - ( \uHk + Q_{h,k}( \uHk )) \|_{L^{\infty}(L^2)} \lesssim_T H^3 + e_{\mbox{\tiny \rm disc}}^{\mbox{\tiny$(1)$}}(h);
\end{align}
if the data is {\rm well-prepared and compatible of order $3$}:
\begin{align}
\label{main-result-est-2b}
\| \partial_t\ueps - \partial_t (\uHk + Q_{h,k}(\uHk)) \|_{L^{\infty}(L^2)} 
+ \| \ueps -  (\uHk + Q_{h,k}( \uHk )) \|_{L^{\infty}(H^1)} \lesssim_T H + e_{\mbox{\tiny \rm disc}}^{\mbox{\tiny$(2)$}}(h);
\end{align}
if $F \in L^{\infty}(H^1_0)$, if the data is {\rm well-prepared and compatible of order $3$} and
if the initial value in (\ref{semi-discrete-lod-equation}) is picked such that we have $\partial_t (\uHk+Q_{h,k}(\uHk))( \cdot , 0 ) = \pi_{H,k}^{\ms}(g)$:
\begin{align}
\label{main-result-est-3b}
\| \partial_t\ueps - \partial_t ( \uHk + Q_{h,k}( \uHk )) \|_{L^{\infty}(L^2)} 
+ \| \ueps -  ( \uHk + Q_{h,k}( \uHk )) \|_{L^{\infty}(H^1)} \lesssim_T H^2 + e_{\mbox{\tiny \rm disc}}^{\mbox{\tiny$(2)$}}(h).
\end{align}
Here, the fine scale discretization errors $e_{\mbox{\tiny \rm disc}}^{\mbox{\tiny$(1)$}}(h)$ and $e_{\mbox{\tiny \rm disc}}^{\mbox{\tiny$(2)$}}(h)$ are given by
$$e_{\mbox{\tiny \rm disc}}^{\mbox{\tiny$(1)$}}(h) := \| \ueps - \pi_h(\ueps) \|_{L^{\infty}(L^2)} + \| \partial_t \ueps - \pi_h(\partial_t \ueps) \|_{L^1(L^2)}$$
and
$$
e_{\mbox{\tiny \rm disc}}^{\mbox{\tiny$(2)$}}(h) := \| \partial_t \ueps - \pi_h(\partial_t \ueps) \|_{L^{\infty}(L^2)} 
+ \| \partial_{tt} \ueps - \pi_h(\partial_{tt} \ueps) \|_{L^1(L^2)} +  \|\ueps - \pi_h(\ueps) \|_{L^{\infty}(H^1)},
$$
where $\pi_h$ is the elliptic-projection on $V_h$ (cf. (\ref{elliptic-projection-pi_h})). Note that $e_{\mbox{\tiny \rm disc}}^{\mbox{\tiny$(i)$}}(h)$ will only yield optimal orders in $h$, if $\ueps$ is sufficiently regular with respect to the spatial variable.
\end{theorem}
As for the $G-$convergence setting we have that if $g\in H^1_0(\Omega)$, $\partial_t F \in L^2(0,T,L^2(\Omega))$ and $f=0$, the estimate in Theorem \ref{apriori-lod-homogenization} can be improved to
\begin{align}
\label{main-result-est-1-sharp_bis}
\| \ueps - \uHk \|_{L^{\infty}(L^2)} \lesssim_T H + e_{\mbox{\tiny \rm disc}}^{\mbox{\tiny$(1)$}}(h).
\end{align}
The proof is analogous to the proof of Lemma \ref{lemma-apriori-hom} presented later.

A proof of Theorem \ref{apriori-semidiscrete} including refined estimates (i.e. estimates where all dependencies on $\ueps$ and $k$ are worked out in detail) is given in Section \ref{section:proofs}.

Our last main result is an optimal $L^{\infty}(L^2)$  error estimate for the Crank-Nicholson version of the multiscale method. In order to obtain the optimal convergence rates with regard to the time step size, we observe that we need slightly higher regularity assumptions than in the semi-discrete case.
\begin{theorem}[A priori error estimates for the Crank-Nicolson fully-discrete method]\label{apriori-fullydiscrete}$\\$
Assume (H0) and that the data is {\rm well-prepared and compatible of order $3$} in the sense of Definition \ref{def-of-comp-condition}. Beside this, let the notation from Theorem \ref{apriori-semidiscrete} hold true and let $u_{H,\triangle t}:=\uHtk$ be the fully discrete numerically homogenized approximation as in Definition \ref{definition-fully-discrete-method}. Then there exists a generic constant $C_\theta$ (i.e. independent of $H$, $h$ and $\eps$) such that if $k \ge C_{\theta} |\ln(H)|$ it holds
\begin{align}
\label{main-result-est-2-fully-discrete}\max_{0 \le n \le J} \| (\ueps - (u_{H,\triangle t} + Q_{h,k}(u_{H,\triangle t})))(\cdot, t^n ) \|_{L^2(\Omega)} \lesssim_T H^2 + \triangle t^2 + e_{\mbox{\tiny \rm disc}}(h).
\end{align}
If furthermore $F \in L^{\infty}(H^1_0)$ and $\partial_t F \in L^2(H^1_0)$, we obtain the improved corrector estimate
\begin{align}
\label{main-result-est-3-fully-discrete} \max_{0 \le n \le J}  \| (\ueps - (u_{H,\triangle t} + Q_{h,k}(u_{H,\triangle t})))(\cdot, t^n )  \|_{L^2(\Omega)} \lesssim_T H^3 + \triangle t^2 + e_{\mbox{\tiny \rm disc}}(h).
\end{align}
Here, the fine scale discretization error $e_{\mbox{\tiny \rm disc}}(h)$ is given by
$$e_{\mbox{\tiny \rm disc}}(h) := \| \ueps - \pi_h(\ueps) \|_{L^{\infty}(L^2)} + \| \partial_t \ueps - \pi_h(\partial_t \ueps) \|_{L^2(L^2)}.$$
The convergence rates in $h$ hence dependent on a higher order space regularity of $\ueps$.
\end{theorem}
The theorem is proved at the end of Section \ref{section:proofs}.
Note that if we are in a situation, where the data is not well-prepared (i.e. \eqref{bound-well-prep} does not hold) not only the mesh size needs to be small enough, but also the time step size $\triangle t$ requires a resolution condition such as $\triangle t \lesssim \eps$. This will become obvious in the proof of Theorem \ref{apriori-fullydiscrete}.

\subsection{Survey on other multiscale methods for the wave equation}
\label{subsection-surbey-multiscale-methods}

The number of existing multiscale methods for the wave equation is rather small, compared to the number of multiscale methods that exist for other types of equations. Subsequently we give a short survey on existing strategies to put our method into perspective. 

One way of realizing numerical homogenization is to use the framework of the Heterogeneous Multiscale Method (HMM) (cf. \cite{Abd05b,Abd06a,Abd11b,EE03,HeO09}). The method is based on the idea to predict an effective limit problem of (\ref{wave-equation-strong}) for $\eps \rightarrow 0$. This can be achieved by solving local problems in sampling cells (typically called cell problems) and to extract effective macroscopic properties from the corresponding cell solutions. In some cases it can be explicitly shown that this strategy in fact yields the correct limit problem for $\eps \rightarrow 0$. The central point of the method is that the cell problems are very small and systematically distributed in $\Omega$, but do not cover $\Omega$. This makes the method very cheap. For the wave equation, an HMM based on Finite Elements was proposed and analyzed in \cite{AbG11}. An HMM based on Finite Differences can be found in \cite{EHR11}. Since the classical homogenized model is known to fail to capture long time dispersive effects (cf. \cite{Lam11b}) another effective model is needed for longer times. Solutions for this problem by a suitable model adaptation in the HMM context can be found in \cite{AGS13,AGS14,EHR12}. The advantage of the HMM framework is that it allows to construct methods that do not have to resolve the fine scale globally, allowing for a computational cost proportional to the degrees of freedom of the macroscopic mesh.
But it requires scale separation and the cell problems must sample the microstructure sufficiently well.
In many applications, especially in material sciences, these assumptions are typically well justified, in geophysical applications on the other hand, they might be often problematic. In this work, we hence focus on the latter case, where the HMM might not be applicable.

Beside the multiscale character of the problem, one of the biggest issues is the typically missing space regularity of the solution. In realistic applications, the propagation field $\aeps$ is discontinuous. For instance in geophysics or seismology, the waves propagate through a medium that consists of different, heterogeneously distributed types of material (e.g. different soil or rock types). Hence, the properties of the propagation field cannot change continuously. This typically also involves a high contrast. The missing smoothness of $\aeps$ directly influences the space regularity of the solution $\ueps$ which is often not higher than $L^{\infty}(0,T;H^1(\Omega))$. As a consequence, the convergence rates of standard Finite Element methods deteriorate besides being very costly. 

To overcome these issues (multiscale character and missing regularity of $\ueps$), Owhadi and Zhang \cite{OwZ08} proposed an interesting multiscale method based on a harmonic coordinate transformation $G$. The method is only analyzed for $d=2$, but it is also applicable for higher dimensions. The components of $G=(G_1,\dots,G_d)$ are defined as the weak solutions of an elliptic boundary value problem $\nabla ( \aeps \nabla G_i ) = 0$ in $\Omega$ and $G_i(x)=x_i$ on $\partial \Omega$. Under a so called Cordes-type condition (cf. \cite[Condition 2.1]{OwZ08}) the authors managed to prove a compensation theorem saying that the solution in harmonic coordinates yields in fact the desired space-regularity. More precisely, they could show that $(\ueps \circ G) \in L^{\infty}(0,T;H^2(\Omega))$ and furthermore that 
$$ \| \ueps \circ G \|_{L^{\infty}(0,T;H^2(\Omega))} \le 
C(F,g) + C \| \partial_{tt} \ueps(\cdot ,0) \|_{L^2(\Omega)},
$$
where $C(F,g)$ and $C$ are constants depending on the data functions, but not on the variations of $\aeps$. Consequently, by using the equality $\partial_{tt} \ueps(\cdot ,0) = \nabla \cdot ( \aeps \nabla f ) - F(\cdot ,0)$, the $L^{\infty}(H^2)$-norm of $\ueps \circ G$ can be bounded {\it independently of the oscillations} of $\aeps$ if the choice of the initial value is such that $\| \nabla \cdot ( \aeps \nabla f ) \|_{L^2(\Omega)}$ can be bounded independent of $\eps$, i.e. if the initial value is well-prepared. Note that $\| \ueps \|_{L^{\infty}(0,T;H^2(\Omega))}$ (if it even exists) is normally proportional to the $W^{1,\infty}$-norm of $\aeps$ (if it exists), which is the reason why classical finite elements cannot converge unless this frequency is resolved by the mesh. The harmonically transformed solution of the wave equation does not suffer from this anymore. With this key feature, an adequate analysis (and corresponding numerics) can be performed in an harmonically transformed finite element space, allowing optimal orders of convergence. The method has only two drawbacks: the approximation of the harmonic coordinate transformation $G$ and the validity of the Cordes-type condition. Even though the Cordes-type condition can be hard to verify in practice, the numerical experiments given in \cite{OwZ08} indicate that the condition might not be necessary for a good behavior of the method. The approximation of the harmonic coordinate transformation $G$ on the other hand can become a real issue, since it involves the solution of $d$ global fine-scale problems. This is an expensive one-time overhead. Furthermore, spline spaces are needed and it is not clear how the analytically predicted results change, when $G$ is replaced by a numerical approximation $G_h$. Compared to \cite{OwZ08}, our method has therefore the advantage that it does not involve to solve global fine scale problems and relies on localized classical P1-finite element spaces.

Another multiscale method applicable to the wave equation was also presented by Owhadi and Zhang in \cite{OwZ11}. Here a multiscale basis is assembled by localizing a certain transfer property (which can be seen as an alternative to the aforementioned harmonic coordinate transformation). In this approach, the number of local problems to solve is basically the same as for our method. However, the local problems require finite element spaces consisting of certain $C^1$-continuous functions. Furthermore, the diameter of the localization patches must at least be of order $\sqrt{H} |\ln(H)|$ to guarantee an optimal linear convergence rate for the $H^1$-error, whereas our approach only requires $H |\ln(H)|$.

The Multiscale Finite Element Method using Limited Global Information 
by Jiang et al. \cite{JEG10,JiE12} can be seen as a general framework that also covers the harmonic coordinate transformation approach by Owhadi and Zhang. The central assumption for this method is the existence of a number of {\it known} global fields $G_1, \dots, G_N$ and an unknown smooth function $H=H(G_1,\dots,G_N,t)$ such that the error $e=\ueps - H(G_1,\dots,G_N,t)$ has a small energy. Based on the size of this energy, an a priori error analysis can be performed. The components of the harmonic coordinate transformation $G$ are an example for global fields that fit into the framework. Other (more heuristic) choices are possible (cf. \cite{JEG10,JiE12}), but equally expensive as computing the harmonic transformation $G$. The drawback of the method is hence the same as for the Owhadi-Zhang approach: the basic assumption on the existence of global fields can be hard to verify and even if it is known to be valid, there is an expensive one-time overhead in computing them with a global fine scale computation.

Excluding the HMM approach, we want to stress that each of the above multiscale methods is only guaranteed to converge in the regime $H>\eps$ if the data is well-prepared in the sense of Definition \ref{def-of-comp-condition}. There are no available results with respect to arbitrary initial data.

With regard to the previous discussions, our multiscale method proposed in Definition \ref{definition-fully-discrete-method} has the following benefits. The method does not require additional assumptions on scale separation or regularity of $\aeps$ and it does not involve one-time-overhead computations on the full fine scale. Furthermore, the method is guaranteed to converge even for not well-prepared initial values. On the other hand, if the initial-values are well-prepared, the method is independent of the homogenization setting and yields significantly improved convergence rates, even in $W^{1,\infty}(L^2)$ and $L^{\infty}(H^1)$, without exploiting any higher space regularity than $H^1$.

\section{Proofs of the main results}
\label{section:proofs}
This section is devoted to the proof of Theorem \ref{apriori-semidiscrete}. First, we derive some general error estimates in Subsection \ref{subsection-abstract-estimates}. In Subsection \ref{subsection-well-prepared-estimates} we analyze the case of well-prepared initial values,
and finally prove in Subsection \ref{subsection-G-conv-estimates} the convergence results without any assumption on the initial data but the one needed for the well-posedness of the wave equation.
\subsection{Abstract error estimates}\label{subsection-abstract-estimates}
Before we can start with proving the a priori error estimates, we present two lemmata. The first result can be found in \cite{Car99,MaP14}:
\begin{lemma}[Properties of the interpolation operator]
\label{lemma-properties-clement-op}
The interpolation operator $I_H : H^1_0(\Omega) \rightarrow V_H$ from (\ref{def-weighted-clement}) has the following properties:
\begin{align}
\label{lemma-properties-clement-op-est}\| v - I_H (v) \|_{L^2(\Omega)} + H \| v - I_H (v) \|_{H^1(\Omega)} \le C_{I_H} H \| v \|_{H^1(\Omega)},
\end{align}
for all $v \in H^1_0(\Omega)$. Here, $C_{I_H}$ denotes a generic constant, that only depends on the shape regularity of the elements of $\T_H$. Furthermore, the restriction $I_H\vert_{V_H} : V_H \rightarrow V_H$ is an isomorphism on $V_H$, with $(I_H\vert_{V_H})^{-1}$ being $H^1$-stable.
\end{lemma}

Observe that $((I_H\vert_{V_H})^{-1} \circ I_H) \vert_{V_H} = \mbox{\rm Id}$. On the other hand $((I_H\vert_{V_H})^{-1} \circ I_H) \vert_{W_h} = 0$. Hence, for any $v_h=v_H +w_h \in V_h=V_H \oplus W_h$ with $v_H \in V_H$ and $w_h \in W_h$ we have $((I_H\vert_{V_H})^{-1} \circ I_H)(v_h)=v_H=P_H(v_h)$ and therefore
\begin{align}
\label{L2-projection-by-interpolation} ((I_H\vert_{V_H})^{-1} \circ I_H) \vert_{V_h} = P_H \vert_{V_h}.
\end{align}
Furthermore we have the equation
\begin{align}
\label{equation-for-projection}\pi_{H,k}^{\ms}(v) = (P_H \circ \pi_{H,k}^{\ms})(v) + (Q_{h,k} \circ P_H \circ \pi_{H,k}^{\ms})(v) \qquad \mbox{for all } v \in H^1_0(\Omega).
\end{align}
The next proposition quantifies the decay of the local correctors.
\begin{proposition}
\label{proposition-stability-estimates}
Let assumptions (H0) be fulfilled and the corrector operators defined according to Definition \ref{definition:localized:ms:space} for some $k \in \mathbb{N}_{>0}$. Then there exists a generic constant $0<\theta<1$ (independent of $H$, $h$ or $\eps$) such that
we have the estimate
\begin{align}
\label{decay-of-correctors}\| \nabla (Q_{h,k} - Q_{h,\Omega})(v_H) \|_{L^2(\Omega)} \lesssim k^{d/2}
\theta^{k} \| \nabla v_H \|_{L^2(\Omega)}.
\end{align}
for all $v_H \in V_H$.
Furthermore, the operator $Q_{h,k}$ is $H^1$-stable on $V_H$ and the operator $(P_{H} \circ \pi_{H,k}^{\ms})$ is $H^1$-stable on $H^1_0(\Omega)$, i.e.
\begin{align}
\label{stability-estimates-for-various-ops}\forall v_H \in V_H: \quad  \hspace{29pt} \| Q_{h,k}(v_H) \|_{H^1(\Omega)} &\lesssim \| v_H \|_{H^1(\Omega)}
\quad \mbox{and}\\
\nonumber \forall v \in H^1_0(\Omega): \quad \| (P_{H} \circ \pi_{H,k}^{\ms})(v) \|_{H^1(\Omega)} &\lesssim \| v \|_{H^1(\Omega)}.
\end{align}
\end{proposition}

\begin{proof}
Estimate (\ref{decay-of-correctors}) was proved in \cite{HeM14}. It is hence sufficient to show \eqref{stability-estimates-for-various-ops}. Since $Q_{h,\Omega}$ is obviously $H^1$-stable on $V_H$ and since $k^{d/2}\theta^k$ is monotonically decreasing for growing $k$, the $H^1$-stability of $Q_{h,k}$ follows directly from (\ref{decay-of-correctors}). The elliptic projection $\pi_{H,k}^{\ms}$ is also obviously $H^1$-stable. 
Finally, the $H^1$-stability of the $L^2$-projection $P_H$ on quasi-uniform meshes (as assumed for $\T_H$) was e.g. proved in \cite{BaD81,BaY14}. Combining these results gives us the desired $H^1$-stability of $P_{H} \circ \pi_{H,k}^{\ms}$ on $H^1_0(\Omega)$.
\end{proof}

The next lemma gives explicit error estimates for the elliptic projections on $\Vmsk$.

\begin{lemma}
\label{lemma-multiscale-projection}
Let $\ueps$ be the solution of (\ref{wave-equation-weak}) and let the corrector operator $Q_{h,k}$ be given as in Definition \ref{definition:localized:ms:space} for some $k \in \mathbb{N}_{>0}$. Furthermore, let $\pi_{H,k}^{\ms}$ and $\pi_h$ denote the elliptic projections according to (\ref{elliptic-projection-pi_Hk}) and (\ref{elliptic-projection-pi_h}). We further denote the $L^2$-projection of $V_h$ on $V_H$ by $P_H$. The following estimates hold for almost every $t \in [0,T]$.\\
If $\partial_t^i \ueps \in L^1(0,T;H^1(\Omega))$ for $i \in \{0,1,2\}$, then it holds
\begin{eqnarray}
\label{projection-error-estimate-1}\lefteqn{\| (P_H \circ \pi_{H,k}^{\ms})(\partial_t^i \ueps(\cdot,t)) - \partial_t^i \ueps(\cdot,t) \|_{L^2(\Omega)}} \\ 
\nonumber&\lesssim& \| \partial_t^i \ueps(\cdot,t) - \pi_h(\partial_t^i \ueps(\cdot,t)) \|_{L^2(\Omega)} + (H + \theta^k k^{d/2}) \| \partial_t^i \ueps(\cdot,t) \|_{H^1(\Omega)}.
\end{eqnarray}
Assume that $i \in \{0,1,2\}$ and $s,m \in \{0,1\}$. If $\partial_t^i \ueps \in L^1(0,T;H^1(\Omega))$ and $\partial_t^{2+i}  \ueps, \partial_t^{i} F \in L^1(0,T;H^s(\Omega))$ it holds
\begin{eqnarray}
\label{projection-error-estimate-3-new-1}\lefteqn{\| \pi_{H,k}^{\ms}(\partial_t^i \ueps(\cdot,t)) - \partial_t^i \ueps(\cdot,t) \|_{H^m(\Omega)} \lesssim \| \partial_t^i \ueps (\cdot,t)- \pi_h(\partial_t^i \ueps(\cdot,t)) \|_{H^m(\Omega)} }\\
\nonumber&\enspace& \enspace + \left(H^{2+s-m} + 
\left(k^{d/2} \theta^{k}\right)^{p(s,m)}
\right) \left( \| \partial_t^{2+i}  \ueps(\cdot,t) -  \partial_t^{i} F(\cdot,t) \|_{H^s(\Omega)} + \|  \partial_t^{i} \ueps(\cdot,t) \|_{H^1(\Omega)} \right),
\end{eqnarray}
where
$p(s,m):= \left(\frac{s+2}{s+1}\right)^{1-m}\ge 1$.

\end{lemma}

\begin{proof}
{\it Error estimate under low regularity assumptions - (\ref{projection-error-estimate-1}).} We use an Aubin-Nitsche duality argument for some arbitrary $v \in L^1(0,T;H^1_0(\Omega))$. Let us define $e_{H,k}:= \pi_{H,k}^{\ms}(v) - \pi_h(v)$. We regard the dual problem: find $z_h \in L^1(0,T;V_h)$ with
\begin{align}
\label{dual-prob-1}\beps(w_h ,z_h\dott ) = (e_{H,k}\dott ,w_h)_{L^2(\Omega)} \qquad \mbox{for all } w_h \in V_h, \quad \mbox{for a.e. } t \in (0,T)
\end{align}
and the dual problem in the multiscale space: find $z_{H,k}^{\ms} \in \Vmsk$ with
\begin{align}
\label{dual-prob-2}\beps(w^{\ms},z_{H,k}^{\ms}\dott ) = (e_{H,k} \dott,w^{\ms})_{L^2(\Omega)} \qquad \mbox{for all } w^{\ms} \in \Vmsk, \quad \mbox{for a.e. } t \in (0,T).
\end{align}
Obviously we have $\beps(w^{\ms},(z_h-z_{H,k}^{\ms})\dott)=0$ for all $w^{\ms} \in \Vmsk$ and for almost every $t \in [0,T]$. This implies that $(z_h-z_{H,\Omega}^{\ms})\dott$ is in the $\beps(\cdot,\cdot)$-orthogonal complement of $\Vms$ (for almost every $t$), hence it is in the kernel of the quasi-interpolation operator $I_H$. Omitting the $t$-dependency, we obtain
\begin{eqnarray}
\label{dual-estimate-1}\nonumber \beps( z_h-z_{H,\Omega}^{\ms}, z_h-z_{H,\Omega}^{\ms} ) &=& (e_{H,k} ,z_h-z_{H,\Omega}^{\ms} )_{L^2(\Omega)} \\
\nonumber &=& (e_{H,k} , (z_h-z_{H,\Omega}^{\ms}) - I_H( z_h-z_{H,\Omega}^{\ms}) )_{L^2(\Omega)} \\
&\overset{(\ref{lemma-properties-clement-op-est})}{\lesssim}& H \| e_{H,k} \|_{L^2(\Omega)} \| z_h-z_{H,\Omega}^{\ms} \|_{H^1(\Omega)}.
\end{eqnarray}
Next, let us define the energy
$$E(v_H):=\beps(z_h - v_H - Q_{h,k}(v_H), z_h - v_H - Q_{h,k}(v_H)) \quad \mbox{for } v_H \in V_H$$
and let
us write $z_{H,\Omega}^{\ms}=z_{H,\Omega} + Q_{h,\Omega}(z_{H,\Omega})$ and $z_{H,k}^{\ms}=z_{H,k} + Q_{h,k}(z_{H,k})$ with $z_{H,\Omega},z_{H,k} \in V_H$. Since we have 
$$\beps( z_h - z_{H,k} - Q_{h,k}(z_{H,k}), v_H + Q_{h,k}(v_H))=0 \qquad \mbox{for all } v_H \in V_H,$$
we know that this is equivalent to the fact that $z_{H,k} \in V_H$ must minimize the energy $E(\cdot)$ on $V_H$.
Hence
\begin{eqnarray}
\nonumber\| z_h - z_{H,k} - Q_{h,k}(z_{H,k}) \|_{H^1(\Omega)} &\lesssim& \| z_h - z_{H,\Omega} - Q_{h,k}(z_{H,\Omega}) \|_{H^1(\Omega)}\\
\nonumber
&\le& \| z_h - z_{H,\Omega}^{\ms} \|_{H^1(\Omega)} + \| (Q_{h,\Omega} - Q_{h,k})(z_{H,\Omega}) \|_{H^1(\Omega)} \\
\nonumber
&\overset{(\ref{dual-estimate-1}),(\ref{decay-of-correctors})}{\lesssim}& H \| e_{H,k} \|_{L^2(\Omega)} + \theta^k k^{d/2} \| Q_{h,\Omega}(z_{H,\Omega})
\|_{H^1(\Omega)}\\
\label{newstep-trolu}&\overset{(\ref{stability-estimates-for-various-ops})}{\lesssim}& (H + \theta^k k^{d/2}) \| e_{H,k} \|_{L^2(\Omega)}.
\end{eqnarray}
Note that in the last step, we used $\| z_{H,\Omega}\|_{H^1} = 
\| P_{H} (z_{H,\Omega})\|_{H^1}
= \| P_{H} ( z_{H,\Omega} + Q_{h,\Omega}(z_{H,\Omega}) )\|_{H^1}
=  \| P_{H} ( z_{H,\Omega}^{\ms} )\|_{H^1}$, which together with
\eqref{stability-estimates-for-various-ops} (i.e. the $H^1$-stability of $P_H$ on quasi-uniform grids) yields
$$
 \| Q_{h,\Omega}(z_{H,\Omega})
\|_{H^1(\Omega)} \lesssim \| z_{H,\Omega}\|_{H^1}
=  \| P_{H} ( z_{H,\Omega}^{\ms} )\|_{H^1} 
\lesssim  \|z_{H,\Omega}^{\ms}\|_{H^1} 
\lesssim \| e_{H,k} \|_{L^2(\Omega)}.
$$
As a direct consequence of \eqref{newstep-trolu}, using $\beps(e_{H,k}, z_{H,k}^{\ms})=0$ (combining (\ref{elliptic-projection-pi_h}) and (\ref{elliptic-projection-pi_Hk}) for test functions in $\Vmsk$)
\begin{align}
\label{L2-estimate-for-general-v-prestage}\| e_{H,k} \|_{L^2(\Omega)}^2 = \beps( e_{H,k} , z_h ) = \beps( e_{H,k} , z_h -  z_{H,k}^{\ms}) \lesssim  \| e_{H,k} \|_{H^1(\Omega)} (H + \theta^k k^{d/2}) \| e_{H,k} \|_{L^2(\Omega)}.
\end{align}
The bound $\| e_{H,k} \|_{H^1(\Omega)} \lesssim \| v \|_{H^1(\Omega)}$ and $\pi_{H,k}^{\ms}(v) = (P_H \circ \pi_{H,k}^{\ms})(v) + (Q_{h,k} \circ P_H \circ \pi_{H,k}^{\ms})(v)$ conclude the estimate
\begin{eqnarray*}
\| (P_H \circ \pi_{H,k}^{\ms})(v) - v \|_{L^2(\Omega)} &\le& 
\| \pi_{H,k}^{\ms}(v) - v \|_{L^2(\Omega)} + \| (Q_{h,k} \circ P_H \circ \pi_{H,k}^{\ms})(v) \|_{L^2(\Omega)}\\
&\overset{(\ref{lemma-properties-clement-op-est}),(\ref{stability-estimates-for-various-ops}),{(\ref{L2-estimate-for-general-v-prestage})}}{\lesssim}& \| v - \pi_h(v) \|_{L^2(\Omega)} + (H + \theta^k k^{d/2}) \| v \|_{H^1(\Omega)} + H \| v \|_{H^1(\Omega)}.
\end{eqnarray*}
Hence for all $v \in H^1_0(\Omega)$
\begin{align}
\label{L2-estimate-for-general-v}
\| (P_H \circ \pi_{H,k}^{\ms})(v) - v \|_{L^2(\Omega)} + \| \pi_{H,k}^{\ms}(v) - v \|_{L^2(\Omega)} 
\lesssim \| v - \pi_h(v) \|_{L^2(\Omega)} + (H + \theta^k k^{d/2}) \| v \|_{H^1(\Omega)}.
\end{align}
The results follows with $v=\partial_t^i \ueps(\cdot,t)$.\\
{\it Error estimate under high regularity assumptions - (\ref{projection-error-estimate-3-new-1}).} For the next estimate, we restrict our considerations to the solution $\ueps$ of (\ref{wave-equation-weak}). Let the regularity assumptions of the lemma hold true and let us introduce the simplifying notation
\begin{align*}
\veps := \partial_t^i \ueps \quad \mbox{and} \quad \bar{F}:=  \partial_t^i F.
\end{align*}
We observe that $\veps$ solves the equation
\begin{align*}
( \partial_{tt} \veps(\cdot,t), w )_{L^2(\Omega)} + \beps(\veps(\cdot,t), w ) = (\bar{F}(\cdot,t) ,w )_{L^2(\Omega)}
\end{align*}
for all $w \in H^1_0(\Omega)$, for almost every $t \in (0,T)$. By the definition of projections, we have
\begin{align*}
\beps( (\pi_{H,\Omega}^{\ms}(\veps) - \pi_h(\veps))(\cdot,t) , w ) = 0  \qquad \mbox{for all } w \in \Vmsk, \quad \mbox{for almost every } t \in (0,T).
\end{align*}
We conclude $(\pi_{H,\Omega}^{\ms}(\veps) - \pi_h(\veps))(\cdot,t) \in W_h$ for almost every $t$ and in particular
\begin{align}
I_H((\pi_{H,\Omega}^{\ms}(\veps) -  \pi_h(\veps))(\cdot,t)) = 0 \quad \mbox{for almost every } t \in (0,T).
\end{align}
Furthermore, with the notation $\pi_{H,k}^{\ms}( \veps )=v_{H,k}+ Q_{h,k}(v_{H,k})$ and $\pi_{H,\Omega}^{\ms}(\veps)=v_{H,\Omega}+ Q_{h,\Omega}(v_{H,\Omega})$ we have again that $v_{H,k}(\cdot,t)\in V_H$ minimizes the
energy
$$E(\Phi_H):=\beps(\pi_h(\veps(\cdot,t) - \Phi_H(\cdot,t) - Q_{h,k}(\Phi_H)(\cdot,t),\pi_h(\veps(\cdot,t) - \Phi_H(\cdot,t) - Q_{h,k}(\Phi_H)(\cdot,t))$$
for $\Phi_H \in V_H$ and therefore
\begin{align*}
\| \pi_{H,k}^{\ms}(\veps(\cdot,t)) - \pi_h(\veps(\cdot,t)) \|_{H^1(\Omega)} 
&= \| v_{H,k}(\cdot,t)+ Q_{h,k}(v_{H,k})(\cdot,t) - \pi_h(\veps(\cdot,t)) \|_{H^1(\Omega)} \\
&\lesssim \| v_{H,\Omega}(\cdot,t)+ Q_{h,k}(v_{H,\Omega})(\cdot,t) - \pi_h(\veps(\cdot,t)) \|_{H^1(\Omega)}.
\end{align*}
For brevity, let us from now on leave out the $t$-dependency in the functions for the rest of the proof.
Hence, we obtain in the same way as for the low regularity estimate
\begin{eqnarray}
\label{general-estimate-pi-H-k}\nonumber\lefteqn{\| \pi_{H,k}^{\ms}(\veps) - \pi_h(\veps) \|_{H^1(\Omega)}}\\
\nonumber&\le& \| u_{H,\Omega}+ Q_{h,\Omega}(v_{H,\Omega}) - \pi_h(\veps) \|_{H^1(\Omega)} + \| ( Q_{h,\Omega} - Q_{h,k})(v_{H,\Omega}) \|_{H^1(\Omega)} \\
\nonumber&\lesssim& \| v_{H,\Omega}+ Q_{h,\Omega}(v_{H,\Omega}) - \pi_h(\veps) \|_{H^1(\Omega)} + k^{d/2} \theta^k \| \pi_{H,\Omega}^{\ms}(\veps) \|_{H^1(\Omega)}\\
&\lesssim& \| \pi_{H,\Omega}^{\ms}(\veps)  - \pi_h(\veps) \|_{H^1(\Omega)} + k^{d/2} \theta^k \| \veps \|_{H^1(\Omega)},
\end{eqnarray}
where we used again the $H^1$-stability of $P_H$ via the equation $v_{H,\Omega}= P_H(v_{H,\Omega}+ Q_{h,\Omega}(v_{H,\Omega})) = (P_H \circ \pi_{H,\Omega}^{\ms})(\veps)$.
We next estimate the term $\| \pi_{H,\Omega}^{\ms}(\veps)  - \pi_h(\veps) \|_{H^1(\Omega)}$ in this estimate. For this, we use the equality
\begin{align}
\label{L2-quasi-orthogonality}( v , w_h )_{L^2(\Omega)} = ( v - I_H(v) , w_h - I_H(w_h) )_{L^2(\Omega)} \qquad \mbox{for all } v \in L^2(\Omega), w_h \in W_h.
\end{align}
This equation holds because of $I_H(w_h)=0$ for all $w_h \in W_h$ and $(v_H,w_h)_{L^2(\Omega)}=0$ for all $v_H \in V_H$ (because $w_h$ is in the kernel of the $L^2$-projection). With that we obtain
\begin{eqnarray*}
\lefteqn{\beps( \pi_{H,\Omega}^{\ms}(\veps)  - \pi_h(\veps) \hspace{2pt},\hspace{2pt} \pi_{H,\Omega}^{\ms}(\veps)  - \pi_h(\veps) )}\\
&=&  \beps( \pi_h(\veps)  \hspace{2pt},\hspace{2pt} \pi_h(\veps)-\pi_{H,\Omega}^{\ms}(\veps) ) \\
&=&  \beps( \veps  \hspace{2pt},\hspace{2pt} \pi_h(\veps)-\pi_{H,\Omega}^{\ms}(\veps) )\\
&=& ( \bar{F} - \partial_{tt}  \veps  \hspace{2pt},\hspace{2pt} \pi_h(\veps)-\pi_{H,\Omega}^{\ms}(\veps) )_{L^2(\Omega)} \\
&\overset{(\ref{L2-quasi-orthogonality})}{=}&  \left( \hspace{2pt} (\bar{F} \hspace{-2pt} - \hspace{-2pt} \partial_{tt}  \veps)-I_H(F \hspace{-2pt} - \hspace{-2pt} \partial_{tt}  \veps) \hspace{2pt} , \hspace{2pt} (\pi_h(\veps) \hspace{-2pt} - \hspace{-2pt} \pi_{H,\Omega}^{\ms}(\veps)) - I_H(\pi_h(\veps) \hspace{-2pt} - \hspace{-2pt} \pi_{H,\Omega}^{\ms}(\veps)) \hspace{2pt} \right)_{L^2(\Omega)} \\
&\overset{(\ref{lemma-properties-clement-op-est})}{\lesssim}& 
H^{s+1} \| \bar{F} - \partial_{tt}  \veps \|_{H^s(\Omega)} \| \pi_h(\veps)-\pi_{H,\Omega}^{\ms}(\veps)) \|_{H^1(\Omega)}.
\end{eqnarray*}
Combining this with (\ref{general-estimate-pi-H-k}) we get
\begin{eqnarray}
\label{estimate-for-H1-error-pi-Hk}
\| \pi_{H,k}^{\ms}(\veps) - \pi_h(\veps) \|_{H^1(\Omega)} \lesssim H^{s+1} \| \partial_{tt}  \veps - \bar{F} \|_{H^s(\Omega)} + k^{d/2} \theta^k \| \veps \|_{H^1(\Omega)},
\end{eqnarray}
which proves the estimate (\ref{projection-error-estimate-3-new-1}) for the case $m=1$. Now, we prove the estimate for the case $m=0$ by applying the same Aubin-Nitsche argument as above. Defining $e_{H,k}:= \pi_{H,k}^{\ms}(\veps) - \pi_h(\veps)$ we are looking for $z_h \in L^2(0,T;V_h)$ and $z_{H,k}^{\ms} \in \Vmsk$ that are defined analogously to (\ref{dual-prob-1}) and (\ref{dual-prob-2}).
Hence we get with same strategy as before 
\begin{align*}
\| z_{H,k}^{\ms} - z_h \|_{H^1(\Omega)} \lesssim (H + k^{d/2} \theta^k) \| e_{H,k} \|_{L^2(\Omega)}
\end{align*}
and hence together with (\ref{estimate-for-H1-error-pi-Hk}) and Youngs inequality 
\begin{align*}
&\| e_{H,k} \|_{L^2(\Omega)}^2 = |\beps(e_{H,k}, z_h - z_{H,k}^{\ms})| \\
&\hspace{3pt}\lesssim 
(H^{s+2} + k^{d{(s+2)/(2s+2)}} \theta^{k{(s+2)/(s+1)}}) \left( \| \partial_{tt}  \veps - \bar{F} \|_{H^s(\Omega)} + \| \veps \|_{H^1(\Omega)} \right) \| e_{H,k} \|_{L^2(\Omega)}.
\end{align*}
In total we proved (\ref{projection-error-estimate-3-new-1}) for $m=0$, i.e.
\begin{eqnarray*}
\lefteqn{\| \pi_{H,k}^{\ms}(\veps) - \veps \|_{L^2(\Omega)}}\\
&\lesssim& \| \veps - \pi_h(\veps) \|_{L^2(\Omega)} + (H^{s+2} + k^{d{(s+2)/(2s+2)}} \theta^{k{(s+2)/(s+1)}}) \left( \| \partial_{tt}  \veps - \bar{F} \|_{H^s(\Omega)} + \| \veps \|_{H^1(\Omega)} \right).
\end{eqnarray*}
\end{proof}
\subsection{Estimates for well-prepared initial values}\label{subsection-well-prepared-estimates}

In the next step, we exploit the estimates derived for the projections $\pi_{H,k}^{\ms}$ to bound the error for the numerically homogenized solutions $\uHk$. The lemma is a data-explicit (in particular $\eps$-explicit and $T$-explicit) version of the estimates (\ref{main-result-est-2a}) and (\ref{main-result-est-3a}) in Theorem \ref{apriori-semidiscrete}.
\begin{lemma}
\label{lemma-apriori-semidiscrete}
Assume that (H0) holds and let $s \in \{0,1\}$. If 
$\partial_t \ueps \in L^1(H^1_0)$; $\partial_{tt} \ueps , F \in L^{\infty}(H^s_0)$ and $\partial_{ttt} \ueps , \partial_t F \in L^1(H^s_0)$ it holds
\begin{eqnarray}
\label{lemma-apriori-semidiscrete-est-2}\lefteqn{\| \ueps - (\uHk + Q_{h,k}(\uHk)) \|_{L^{\infty}(L^2)} }\\
\nonumber
&\lesssim& \| \ueps - \pi_h(\ueps) \|_{L^{\infty}(L^2)} + \| \partial_t \ueps - \pi_h(\partial_t \ueps) \|_{L^1(L^2)} \\
\nonumber
&\enspace& + (H^{2+s} + k^{d} \theta^{k{(s+2)/(s+1)}}) \left(\| \ueps \|_{L^{\infty}(H^1)} + \| \partial_t \ueps \|_{L^1(H^1)} + \| \partial_{tt}  \ueps \|_{L^{\infty}(H^s)}  + \| \partial_{ttt}  \ueps \|_{L^1(H^s)} \right. \\
\nonumber
&\enspace& \qquad \left. + \| F \|_{L^{\infty}(H^s)} + \| \partial_t F \|_{L^1(H^s)} \right).
\end{eqnarray}
Recall that the $\lesssim$-notation only contains dependencies on $\Omega$, $d$, $\alpha$, $\beta$ and the shape regularity of $\T_H$, but not on $T$ and $\eps$.
\end{lemma}

The above Lemma proves the first part of Theorem \ref{apriori-semidiscrete}, i.e. estimates (\ref{main-result-est-2a}) and (\ref{main-result-est-3a}), as explained next.\\
\noindent[Proof of estimates (\ref{main-result-est-2a}) and (\ref{main-result-est-3a}) in Theorem \ref{apriori-semidiscrete}] \label{conclusion-3}
Exploiting the time-regularity result presented in Proposition \ref{proposition-time-regularity}, we observe that if $F \in L^{\infty}(H^s_0)\cap W^{1,1}(H^s_0)$ and if the data is {\it well-prepared and compatible of order $2+s$} in the sense of Definition \ref{def-of-comp-condition}, we obtain that $\ueps$ is sufficiently regular for Lemma \ref{lemma-apriori-semidiscrete} to hold. Furthermore we have 
$$\| \ueps \|_{L^{\infty}(H^1)} + \| \partial_t \ueps \|_{L^1(H^1)} + \| \partial_{tt}  \ueps \|_{L^{\infty}(H^s)}  + \| \partial_{ttt}  \ueps \|_{L^1(H^s)} \le C_w$$
independent of $\eps$. 
In order to treat the $\theta$-terms in \eqref{lemma-apriori-semidiscrete-est-2}, we choose
$k:= \frac{(s+1)\ln(H)}{(s+2)\ln(\theta)}(s+2+\delta)$
for some $\delta>0$. This gives us 
$$\theta^{k (s+2)/(s+1)} = \theta^{(s+2+\delta) \ln(H) / \ln(\theta)} =e^{ \ln(\theta) (s+2+\delta) \ln(H) / \ln(\theta)}
= e^{ (s+2+\delta) \ln(H) } = H^{s+2+\delta}.
 $$
For $k$ as above and $\delta>0$ we hence have $k^{d} \theta^{k(s+2)/(s+1)} \lesssim  H^{s+2}$. The constant $C_{\theta}$ in Theorem \ref{apriori-semidiscrete} can hence be chosen as $C_{\theta}=\frac{8}{3 |\ln(\theta)|}$ in the worst case.
This ends the proof of (\ref{main-result-est-2a}) and (\ref{main-result-est-3a}) in Theorem \ref{apriori-semidiscrete}.
\hfill$\square$

\begin{proof}[Proof of Lemma \ref{lemma-apriori-semidiscrete}]
To prove the result, we can follow the arguments of Baker \cite{Bak76}. For the numerically homogenized solution $\uHk$ of (\ref{semi-discrete-lod-equation}), we define $\umsk:=\uHk + Q_{h,k}(\uHk)$. For brevity, we denote $(\cdot,\cdot):=(\cdot,\cdot)_{L^2(\Omega)}$. Furthermore, we use the notation from Lemma \ref{lemma-multiscale-projection} and define the errors
\begin{align*}
e^{\ms}
:= \ueps - \umsk, \quad 
e^{\pi}
:= \ueps - \pi_{H,k}^{\ms}(\ueps) \quad \mbox{and} \quad 
\psi^{\pi}
:= \umsk - \pi_{H,k}^{\ms}(\ueps). 
\end{align*}
Observe that we have for $v \in L^2(0,T;\Vmsk)$ and almost every $t \ge 0$:
\begin{eqnarray*}
\lefteqn{0 = ( \hspace{2pt} \partial_{tt} \psi^{\pi}(\cdot,t) - \partial_{tt} e^{\pi}(\cdot,t) , v \dott \hspace{2pt} ) + \beps( \hspace{2pt} \psi^{\pi}(\cdot,t), v \dott \hspace{2pt} )}\\ 
&=& \partial_{t}  ( \hspace{2pt} \partial_{t} \psi^{\pi}(\cdot,t) - \partial_{t} e^{\pi}(\cdot,t) , v \dott \hspace{2pt} )
- ( \hspace{2pt} \partial_{t} \psi^{\pi}(\cdot,t) - \partial_{t} e^{\pi}(\cdot,t) , \partial_{t} v \dott \hspace{2pt} )
 + \beps( \hspace{2pt} \psi^{\pi}(\cdot,t), v \dott \hspace{2pt} )\\
 &=& - \partial_{t}  ( \hspace{2pt} \partial_{t} e^{\ms}(\cdot,t) , v \dott \hspace{2pt} ) - ( \hspace{2pt} \partial_{t} \psi^{\pi}(\cdot,t) - \partial_{t} e^{\pi}(\cdot,t) , \partial_{t} v \dott \hspace{2pt} )
 + \beps( \hspace{2pt} \psi^{\pi}(\cdot,t), v \dott \hspace{2pt} ).
\end{eqnarray*}
For some arbitrary $0<t_0\le T$ we use the function $v \dott = \int_t^{t_0} \psi^{\pi}(\cdot,s) \hspace{2pt} ds$ in the above equation (and the fact that $\partial_t v = - \psi^{\pi}$) to obtain
\begin{eqnarray*}
\lefteqn{ \frac{1}{2} \frac{d}{dt} \| \psi^{\pi}(\cdot,t) \|_{L^2(\Omega)}^2
-  \frac{1}{2} \frac{d}{dt} \beps \left( \hspace{2pt}  \int_t^{t_0} \psi^{\pi}(\cdot,s) \hspace{2pt} ds ,  \int_t^{t_0} \psi^{\pi}(\cdot,s) \hspace{2pt} ds \hspace{2pt} \right) } \\
&=& \partial_{t}  \left( \hspace{2pt} \partial_{t} e^{\ms}(\cdot,t) , \int_t^{t_0} \psi^{\pi}(\cdot,s) \hspace{2pt} ds \hspace{2pt} \right) + ( \partial_{t} e^{\pi}(\cdot,t) , \psi^{\pi} \dott \hspace{2pt} ).
\end{eqnarray*}
Integration from $0$ to ${t_0}$ yields
\begin{eqnarray*}
\lefteqn{ \frac{1}{2} \| \psi^{\pi}(\cdot,t_0) \|_{L^2(\Omega)}^2 - \frac{1}{2} \| \psi^{\pi}(\cdot,0) \|_{L^2(\Omega)}^2
+  \frac{1}{2} \beps \left( \hspace{2pt}  \int_0^{t_0} \psi^{\pi}(\cdot,s) \hspace{2pt} ds ,  \int_0^{t_0} \psi^{\pi}(\cdot,s) \hspace{2pt} ds \hspace{2pt} \right) } \\
&=& - \left( \hspace{2pt} \partial_{t} e^{\ms}(\cdot,0) , \int_0^{t_0} \psi^{\pi}(\cdot,s) \hspace{2pt} ds \hspace{2pt} \right) + \int_0^{t_0} ( \partial_{t} e^{\pi}(\cdot,t) , \psi^{\pi} \dott \hspace{2pt} ) \hspace{2pt} dt.
\end{eqnarray*}
Hence
\begin{eqnarray*}
\lefteqn{\| \psi^{\pi}(\cdot,t_0) \|_{L^2(\Omega)}^2}\\
&\le& \| \psi^{\pi}(\cdot,0) \|_{L^2(\Omega)}^2 - 2 \left( \hspace{2pt} \partial_{t} e^{\ms}(\cdot,0) , \int_0^{t_0} \psi^{\pi}(\cdot,s) \hspace{2pt} ds \hspace{2pt} \right) + 2 \int_0^{t_0} ( \partial_{t} e^{\pi}(\cdot,t) , \psi^{\pi} \dott \hspace{2pt} ) \hspace{2pt} dt \\
&\le& \| \psi^{\pi}(\cdot,0) \|_{L^2(\Omega)}^2 + 2 \int_0^{t_0} ( \partial_{t} e^{\pi}(\cdot,t) , \psi^{\pi} \dott \hspace{2pt} ) \hspace{2pt} dt \\
&\le& \| \psi^{\pi}(\cdot,0) \|_{L^2(\Omega)}^2 + 
2 \| \partial_{t} e^{\pi} \|_{L^1(0,T;L^2(\Omega))}^2 
+ \frac{1}{2} \| \psi^{\pi} \|_{L^{\infty}(0,T;L^2(\Omega))}^2.
\end{eqnarray*}
By moving the term $\| \psi^{\pi} \|_{L^{\infty}(0,T;L^2(\Omega))}^2$ to the left hand side, we get
\begin{align}
\label{lemma-first-est-befor-tri-ineq}\| \psi^{\pi} \|_{L^{\infty}(0,T;L^2(\Omega))} \lesssim \| \psi^{\pi}(\cdot,0) \|_{L^2(\Omega)} + 
 \| \partial_{t} e^{\pi} \|_{L^1(0,T;L^2(\Omega))}.
\end{align}
However, since $\umsk(\cdot,0)=\pi_{H,k}^{\ms}(f)$, we get $\psi^{\pi}(\cdot,0)=\pi_{H,k}^{\ms}(f)-\pi_{H,k}^{\ms}(\ueps(\cdot,0))=0$. Hence, together with the triangle inequality for $\psi^{\pi}=(\umsk - \ueps) + (\ueps - \pi_{H,k}^{\ms}(\ueps))$, equation (\ref{lemma-first-est-befor-tri-ineq}) implies
\begin{align*}
\| \umsk - \ueps \|_{L^{\infty}(0,T;L^2(\Omega))} \lesssim \| \ueps - \pi_{H,k}^{\ms}(\ueps) \|_{L^{\infty}(0,T;L^2(\Omega))} + 
\| \partial_{t} \ueps - \pi_{H,k}^{\ms}(\partial_{t} \ueps) \|_{L^1(0,T;L^2(\Omega))}.
\end{align*}
Together with Lemma \ref{lemma-multiscale-projection} this finishes the proof of (\ref{lemma-apriori-semidiscrete-est-2}).
\end{proof}

Next, we prove an $\eps$-explicit and $T$-explicit version of the estimates 
from the second part of Theorem \ref{apriori-semidiscrete}.

\noindent[Proof of estimates (\ref{main-result-est-2b}) and (\ref{main-result-est-3b}) in Theorem \ref{apriori-semidiscrete}]
Similarly as for the first part of Theorem \ref{apriori-semidiscrete}, these
are obtained by combining Lemma \ref{lemma-apriori-semidiscrete-H1} below with the regularity statement in Proposition \ref{proposition-time-regularity}.
\hfill$\square$

\begin{lemma}
\label{lemma-apriori-semidiscrete-H1}
Let (H0) be fulfilled, let $s \in \{0,1\}$ and assume 
$\partial_t \ueps \in L^{\infty}(H^1_0)$; $\partial_{ttt} \ueps, \partial_t F \in L^{\infty}(L^2)$; $\partial_{tt} \ueps, F \in L^{\infty}(H^s_0)$; $\partial_{tt} \ueps \in L^1(H^1_0)$; $\partial_{t}^4 \ueps , \partial_{tt} F \in L^1(L^2)$ and $g \in H^1_0(\Omega)$.\\
If $s=0$, it holds
\begin{eqnarray}
\label{lemma-apriori-semidiscrete-est-H1-2}\lefteqn{\| \partial_t \ueps - \partial_t (\uHk + Q_{h,k}(\uHk)) \|_{L^{\infty}(L^2)} 
+ \| \ueps - (\uHk + Q_{h,k}(\uHk)) \|_{L^{\infty}(H^1)}}\\
\nonumber
&\lesssim& \| \partial_t \ueps - \pi_h(\partial_t \ueps) \|_{L^{\infty}(L^2)} 
+ \| \partial_{tt} \ueps - \pi_h(\partial_{tt} \ueps) \|_{L^1(L^2)} +  \|\ueps - \pi_h(\ueps) \|_{L^{\infty}(H^1)} \\
\nonumber
&\enspace& + (H + k^{d/2} \theta^{k}) \left( 
{\sum_{i=0}^1} \| \partial^i_t \ueps \|_{L^{\infty}(H^1)} 
+ \| \partial_{tt} \ueps \|_{L^{\infty}(L^2)}
+ \| \partial_{tt}  \ueps \|_{L^{1}(H^1)} 
+ \| \partial_{ttt} \ueps \|_{L^{\infty}(L^2)} \right)\\
\nonumber&\enspace& + (H + k^{d/2} \theta^{k}) \left( 
 \| \partial_{t}^4  \ueps \|_{L^{1}(L^2)} 
+ {\sum_{i=0}^1} \| \partial_t^i F \|_{L^{\infty}(L^2)}
+ \| \partial_{tt} F \|_{L^1(L^2)}
+ \| g \|_{H^1(\Omega)} \right).
\end{eqnarray}
If $s=1$ and if the initial value in (\ref{semi-discrete-lod-equation}) is picked such that $\partial_t (\uHk+Q_{h,k}(\uHk))( \cdot , 0 ) = \pi_{H,k}^{\ms}(g)$, then we obtain the improved estimate
\begin{eqnarray}
\label{lemma-apriori-semidiscrete-est-H1-3}\lefteqn{\| \partial_t \ueps - \partial_t (\uHk + Q_{h,k}(\uHk)) \|_{L^{\infty}(L^2)} 
+ \| \ueps - (\uHk + Q_{h,k}(\uHk)) \|_{L^{\infty}(H^1)}}\\
\nonumber
&\lesssim& \| \partial_t \ueps - \pi_h(\partial_t \ueps) \|_{L^{\infty}(L^2)} 
+ \| \partial_{tt} \ueps - \pi_h(\partial_{tt} \ueps) \|_{L^1(L^2)} +  \|\ueps - \pi_h(\ueps) \|_{L^{\infty}(H^1)} \\
\nonumber
&\enspace& + (H^2 + k^{d/2} \theta^{k}) \left( 
{\sum_{i=0}^2} \| \partial^i_t \ueps \|_{L^{\infty}(H^1)} 
+ \| \partial_{tt}  \ueps \|_{L^{1}(H^1)} 
+ \| \partial_{ttt} \ueps \|_{L^{\infty}(L^2)} \right)\\
\nonumber&\enspace& + (H^2 + k^{d/2} \theta^{k}) \left( 
 \| \partial_{t}^4  \ueps \|_{L^{1}(L^2)} 
+ \| F \|_{L^{\infty}(H^1)}
+ \| \partial_t F \|_{L^{\infty}(L^2)}
+ \| \partial_{tt} F \|_{L^1(L^2)} \right).
\end{eqnarray}
Again, recall that the $\lesssim$-notation only contains dependencies on $\Omega$, $d$, $\alpha$, $\beta$ and the shape regularity of $\T_H$, but not on $T$ and $\eps$.
\end{lemma}

\begin{proof}
Again, we define the errors
$
e^{\ms}
:= \ueps - \umsk$, $e^{\pi}:= \ueps - \pi_{H,k}^{\ms}(\ueps)$ and $ \psi^{\pi}
:= \umsk - \pi_{H,k}^{\ms}(\ueps)$. 
We only consider the case $\partial_t \umsk( \cdot , 0 ) = P_{H,k}^{\ms}(g)$ (i.e. estimate (\ref{lemma-apriori-semidiscrete-est-H1-2})), the case $\partial_t \umsk( \cdot , 0 ) = \pi_{H,k}^{\ms}(g)$ (i.e. estimate (\ref{lemma-apriori-semidiscrete-est-H1-3})) follows analogously with $\partial_t \psi^{\pi}(\cdot,0)=0$.
By Galerkin orthogonality we obtain for a.e. $t\in[0,T]$
\begin{align*}
(\partial_{tt} e^{\ms}(\cdot,t) , v^{\ms} )_{L^2(\Omega)} + \beps( e^{\ms} (\cdot,t), v^{\ms} ) = 0  \qquad \mbox{for all } v^{\ms} \in \Vmsk
\end{align*}
and hence
\begin{align*}
(\partial_{tt} \psi^{\pi}(\cdot,t) , v^{\ms} )_{L^2(\Omega)} + \beps( \psi^{\pi} (\cdot,t), v^{\ms} ) = 
(\partial_{tt} e^{\pi}(\cdot,t) , v^{\ms} )_{L^2(\Omega)}
\qquad \mbox{for all } v^{\ms} \in \Vmsk.
\end{align*}
Testing with $v^{\ms} = \partial_t \psi^{\pi}$ yields for a.e. $t \in [0,T]$
\begin{eqnarray*}
\frac{1}{2} \frac{\mbox{d}}{\mbox{d}t} \left( 
\| \partial_t \psi^{\pi}(\cdot,t) \|^2_{L^2(\Omega)} + \beps( \psi^{\pi}(\cdot,t),\psi^{\pi}(\cdot,t))
\right) = (\partial_{tt} e^{\pi}(\cdot,t) , \partial_t \psi^{\pi}(\cdot,t) )_{L^2(\Omega)}.
\end{eqnarray*}
By integration from $0$ to $t_0 \le T$ we obtain
\begin{eqnarray*}
\lefteqn{\frac{1}{2} \| \partial_t \psi^{\pi}(\cdot,t_0) \|^2_{L^2(\Omega)} + \frac{\alpha}{2} \| \psi^{\pi}(\cdot,t_0) \|_{H^1(\Omega)}^2}\\
&\le& \frac{1}{2} \| \partial_t \psi^{\pi}(\cdot,0) \|^2_{L^2(\Omega)} + \frac{\beta}{2} \| \psi^{\pi}(\cdot,0) \|_{H^1(\Omega)}^2
+ \int_{0}^{t_0} |(\partial_{tt} e^{\pi}(\cdot,t) , \partial_t \psi^{\pi}(\cdot,t) )_{L^2(\Omega)}| \hspace{2pt} dt.
\end{eqnarray*}
Since we have $\psi^{\pi}(\cdot,0) = \pi_{H,k}^{\ms}(f) -\pi_{H,k}^{\ms}(f) =0$ we get with the Young and the Cauchy-Schwarz inequality
\begin{eqnarray*}
\lefteqn{\frac{1}{2} \| \partial_t \psi^{\pi}(\cdot,t_0) \|^2_{L^2(\Omega)} + \frac{\alpha}{2} \| \psi^{\pi}(\cdot,t_0) \|_{H^1(\Omega)}^2}\\
&\le& \frac{1}{2} \| \pi_{H,k}^{\ms}(g) - P_{H,k}^{\ms}(g) \|^2_{L^2(\Omega)}
+ \| \partial_{tt} e^{\pi} \|_{L^1(0,T;L^2(\Omega))}^2  + \frac{1}{4} \| \partial_t \psi^{\pi} \|_{L^{\infty}(0,T;L^2(\Omega))}.
\end{eqnarray*}
By taking the supremum over all $0\le t_0 \le T$ we obtain
\begin{eqnarray*}
{\| \partial_t \psi^{\pi} \|^2_{L^{\infty}(L^2)} + 2 \alpha \| \psi^{\pi} \|_{L^{\infty}(H^1)}^2}
\le2 \| \pi_{H,k}^{\ms}(g) - P_{H,k}^{\ms}(g) \|^2_{L^2(\Omega)}
+ 4 \| \partial_{tt} e^{\pi} \|_{L^1(L^2)}^2.
\end{eqnarray*}
The term $\| \partial_{tt} e^{\pi} \|_{L^1(L^2)}$ can be treated with Lemma \ref{lemma-multiscale-projection}, equation (\ref{projection-error-estimate-3-new-1}). Hence it only remains to estimate the term $ \| \pi_{H,k}^{\ms}(g) - P_{H,k}^{\ms}(g) \|^2_{L^2(\Omega)}$. Observe that this term vanishes in \eqref{lemma-apriori-semidiscrete-est-H1-3} as $\partial_t \psi^{\pi}(\cdot,0)=0$. For this term, we have
\begin{align*}
 \| \pi_{H,k}^{\ms}(g) - P_{H,k}^{\ms}(g) \|_{L^2(\Omega)} &\le \| \pi_{H,k}^{\ms}(g) - g \|_{L^2(\Omega)} + \| P_{H,k}^{\ms}(g) - g \|_{L^2(\Omega)}\\
 &\le 2 \| \pi_{H,k}^{\ms}(g) - g \|_{L^2(\Omega)} = 2 \| \pi_{H,k}^{\ms}(g) - g - I_H( \pi_{H,k}^{\ms}(g) - g) \|_{L^2(\Omega)} \\
 &\lesssim H \| \pi_{H,k}^{\ms}(g) - g \|_{H^1(\Omega)} \lesssim H \| g \|_{H^1(\Omega)}.
\end{align*}
Hence, the triangle inequality yields
\begin{eqnarray*}
\lefteqn{\| \partial_t e^{\ms} \|_{L^{\infty}(L^2)} + \| e^{\ms} \|_{L^{\infty}(H^1)}}\\
&\lesssim& \|  \partial_t \ueps- \pi_{H,k}^{\ms}(\partial_t \ueps) \|_{L^{\infty}(L^2)} + \| \ueps- \pi_{H,k}^{\ms}( \ueps) \|_{L^{\infty}(H^1)}\\
&\enspace& +  \|  \partial_{tt} \ueps- \pi_{H,k}^{\ms}(\partial_{tt} \ueps) \|_{L^1(L^2)} + H \| g \|_{H^1(\Omega)}.
\end{eqnarray*}
Lemma \ref{lemma-multiscale-projection} finishes the proof.
\end{proof}
The proof of Lemmas \ref{lemma-apriori-semidiscrete}
and \ref{lemma-apriori-semidiscrete-H1} hence finish as explained the
proof of Theorem \ref{apriori-semidiscrete}. 
We conclude this section by proving the fully-discrete estimate stated in Theorem \ref{apriori-fullydiscrete}.
\begin{proof}[Proof of Theorem \ref{apriori-fullydiscrete}]
The first part of the proof is completely analogous to the one presented by Baker \cite[Section 4]{Bak76} for the classical finite element method. With the same arguments, we can show that
\begin{eqnarray*}
\lefteqn{\max_{0 \le n \le J} \| (\ueps - \uHtk - Q_{h,k}(\uHtk) )(\cdot, t^n ) \|_{L^2(\Omega)}}\\
&\lesssim& \max_{0 \le n \le J} \| (\ueps - \pi_{H,k}^{\ms}(\ueps)) (\cdot, t^n ) \|_{L^2(\Omega)} \\
&\enspace& + \| (\uHtk + Q_{h,k}(\uHtk))(\cdot,0) - \pi_{H,k}^{\ms}(\ueps(\cdot,0)) \|_{L^2(\Omega)} \\
&\enspace& + \| (\partial_t \ueps - \pi_{H,k}^{\ms}(\partial_t \ueps)) \|_{L^2(0,T;L^2(\Omega))}  + \triangle t^2 \left( \| \partial_{t}^3 \ueps \|_{L^2(0,T;L^2(\Omega))} + \| \partial_{t}^4 \ueps \|_{L^2(0,T;L^2(\Omega))} \right),
\end{eqnarray*}
where we note that the term $ \| (\uHtk + Q_{h,k}(\uHtk))(\cdot,0) - \pi_{H,k}^{\ms}(\ueps(\cdot,0)) \|_{L^2(\Omega)}$ is equal to zero, since obviously $(\uHtk + Q_{h,k}(\uHtk))(\cdot,0) = \pi_{H,k}^{\ms}(f)$ by definition of the method. 
The last two terms are already readily estimated. It only remains to bound the two terms $\| (\ueps - \pi_{H,k}^{\ms}(\ueps)) (\cdot, t^n ) \|_{L^2(\Omega)}$ and $\| (\partial_t \ueps - \pi_{H,k}^{\ms}(\partial_t \ueps)) \|_{L^2(L^2)}$ using (\ref{projection-error-estimate-3-new-1}). That finishes the proof.
\end{proof}

\subsection{Estimates in the setting of $G$-convergence}\label{subsection-G-conv-estimates}

In this subsection we prove the homogenization result stated in Theorem \ref{apriori-lod-homogenization}. Consequently, we assume that we are in the homogenization setting of $G$-convergence as established in Definition \ref{def-G-convergence} and Theorem \ref{theorem:homogenization:wave:equation}.
Before we can start the proof, we need to introduce an auxiliary problem. Define
$\hatueps \in L^2(0,T;H^1_0(\Omega))$
as the solution to the wave equation 
\begin{align}
\label{wave-equation-weak-f-eps}
\nonumber\langle \partial_{tt} \hatueps \dott, v \rangle + \left( \aeps \nabla \hatueps \dott, \nabla v \right)_{L^2(\Omega)} &= \left( F \dott, v \right)_{L^2(\Omega)},\\
\left( \hatueps( \cdot , 0 ) , v \right)_{L^2(\Omega)} = \left( f^{\eps} , v \right)_{L^2(\Omega)},\quad
\left( \partial_t \hatueps( \cdot , 0 ) , v \right)_{L^2(\Omega)} &= \left( g , v \right)_{L^2(\Omega)}
\end{align}
for all $v \in H^1_0(\Omega)$ and a.e. $t > 0$,
where the $f^{\eps}$ are defined from the initial value of the original wave equation by \eqref{def-f-eps}.
We recall the definition of the homogenization error 
\begin{align}
\label{def-hom-error}
e_{\mbox{\tiny \rm hom}}(\eps)
 := \| \uhom - \ueps \|_{L^{\infty}(L^2)} + \| f - f^{\eps} \|_{L^2(\Omega)}
\end{align}
and define the following fine scale discretization error
\begin{align}
\label{def-fine-disc-error}
&\hat{e}_{\mbox{\tiny \rm disc}}(h) := \\
\nonumber &\quad
\| \hatueps \hspace{-2pt}-\hspace{-2pt} \pi_h(\hatueps) \|_{L^{\infty}(L^2)} 
+ \| \partial_t \hatueps \hspace{-2pt}-\hspace{-2pt} \pi_h(\partial_t \hatueps) \|_{L^2(L^2)} 
+ \| f^{\eps} \hspace{-2pt}-\hspace{-2pt} \pi_h(f^{\eps}) \|_{L^2(\Omega)}
+ \| f \hspace{-2pt}-\hspace{-2pt} \pi_h(f) \|_{L^2(\Omega)}. 
\end{align}
The following lemma is the main ingredient to prove Theorem \ref{apriori-lod-homogenization}.
\begin{lemma}
\label{lemma-apriori-hom}
Assume that (H0) holds. Let furthermore $g\in H^1_0(\Omega)$, $\partial_t F \in L^2(0,T,L^2(\Omega))$ and $\nabla \cdot (\ahom \nabla f) + F(\cdot,0)\in L^2(\Omega)$.
Then it holds
\begin{eqnarray}
\label{lemma-apriori-semidiscrete-est-1}\lefteqn{\| \uhom - \uHk \|_{L^{\infty}(L^2)} \lesssim_T e_{\mbox{\tiny \rm hom}}(\eps) + \hat{e}_{\mbox{\tiny \rm disc}}(h)} \\
\nonumber
&+& ( H + \theta^k k^{d/2}) \left( \| F \|_{W^{1,2}(L^2)} + \| f \|_{H^1(\Omega)} + \| g \|_{H^1(\Omega)} 
+ \| \nabla \cdot (\ahom \nabla f)
\|_{L^2(\Omega)} \right).
\end{eqnarray}
If we replace the elliptic projection $\pi_{H,k}^{\ms}(f)$ in (\ref{semi-discrete-lod-equation}) by the $L^2$-projection $P_{H,k}^{\ms}(f)$, the term $\| f \hspace{-2pt}-\hspace{-2pt} \pi_h(f) \|_{L^2(\Omega)}$ in the definition of $\hat{e}_{\mbox{\tiny \rm disc}}(h)$ can be dropped.
\end{lemma}
\noindent[Proof of estimates \eqref{main-result-est-1} and \eqref{main-result-est-2} in Theorem \ref{apriori-lod-homogenization}.]
We observe that under the assumptions of Theorem \ref{apriori-lod-homogenization} we obviously have
\begin{align*}
\lim_{\eps \rightarrow 0} e_{\mbox{\tiny \rm hom}}(\eps) = 0 \qquad \mbox{and} \qquad \lim_{h \rightarrow 0} \hat{e}_{\mbox{\tiny \rm disc}}(h) = 0.
\end{align*}
Combining this observation with \eqref{lemma-apriori-semidiscrete-est-1} gives \eqref{main-result-est-1}. Next using a triangle inequality
for the term $(u^\eps- \uHk)=(u^\eps- u^0)+(u^0-\uHk)$ together with the estimate \eqref{main-result-est-1} yields \eqref{main-result-est-2}.
\hfill$\square$

\begin{remark}
If we are in the homogenization setting of $G$-convergence and if $h<\eps$, we can assume $e_{\mbox{\tiny \rm hom}}(\eps) + \hat{e}_{\mbox{\tiny \rm disc}}(h) \lesssim H$. This bound resembles the fact the fine grid resolves the microstructures and that the coarse grid is still coarse compared to the speed of the data oscillations in $\aeps$. In this case estimate \eqref{lemma-apriori-semidiscrete-est-1} simplifies to
$$
\| \uhom - \uHk\|_{L^{\infty}(L^2)} \lesssim_T H + \theta^k k^{d/2}.
$$
\end{remark}

\begin{proof}
First, recall that $P_{H,k}^{\ms}$ denotes the $L^2$-projection on $\Vmsk$ and $\pi_{H,k}^{\ms}$ the elliptic projection. In this proof we treat both cases, i.e. $(\uHk+Q_{h,k}(\uHk))( \cdot , 0 ) = \pi_{H,k}^{\ms}(f)$ and $(\uHk+Q_{h,k}(\uHk))( \cdot , 0 ) = P_{H,k}^{\ms}(f)$, at the same time. The main proof consists of six steps. In the first step, we state some properties for the auxiliary problem defined via \eqref{wave-equation-weak-f-eps}. In the second step we split the error $\| \uhom - u_{H} \|_{L^{\infty}(L^2)}$ into several contributions. In the last four steps, these contributions are estimated and combined.

$\\$
{\it Step 1 (auxiliary problem and properties).} Let $f^{\eps} \in H^1_0(\Omega)$ denote the solution of the auxiliary problem \eqref{def-f-eps}. Hence, we have
\begin{align}
\label{energy-f-eps}\| \nabla f^{\eps} \|_{L^2(\Omega)}\le \frac{\beta}{\alpha} \| \nabla f \|_{L^2(\Omega)}.
\end{align}
With this, we use the following notation. By $\ueps$
we denote the solution to the original wave equation \eqref{wave-equation-weak} and by 
$\hatueps$ the solution to the wave equation \eqref{wave-equation-weak-f-eps}, i.e. the wave equation with modified initial value $f^{\eps}$. First, exploiting \eqref{energy-f-eps}, we observe that $\hatueps$ fulfills the classical energy estimate
\begin{align}
\label{data-est-1}
\| \partial_t \hatueps \|_{L^{\infty}(0,T;L^2(\Omega))}
+ \|  \hatueps \|_{L^{\infty}(0,T;H^1(\Omega))}
\lesssim_{T} \| F \|_{L^2(0,T;L^2(\Omega))}  + \| g \|_{L^2(\Omega)} + \| \nabla f \|_{L^2(\Omega)}.
\end{align} 
Furthermore, since by construction $\nabla \cdot (\aeps \nabla f^{\eps}) = \nabla \cdot (\ahom \nabla f) \in L^2(\Omega)$, we can use Proposition \ref{proposition-time-regularity} to verify that
$$
\partial_t \hatueps \in L^{\infty}(0,T;H^1_0(\Omega)) \qquad \mbox{and} \qquad  \partial_{tt} \hatueps  \in L^{\infty}(0,T;L^2(\Omega)).
$$
Consequently we have for all $v\in H^1_0(\Omega)$
\begin{eqnarray*}
\lefteqn{(\partial_{tt} \hatueps(\cdot, 0) , v )_{L^2(\Omega)} = - ( \aeps \nabla \hatueps(\cdot, 0) , \nabla v)_{L^2(\Omega)} + (F(\cdot, 0) , v )_{L^2(\Omega)}}\\
&=& - ( \aeps \nabla f^{\eps} , \nabla v)_{L^2(\Omega)} + (F(\cdot, 0) , v )_{L^2(\Omega)}
= (F(\cdot, 0) + \nabla \cdot (\ahom \nabla f ) , v )_{L^2(\Omega)}
\end{eqnarray*}
and therefore $\| \partial_{tt} \hatueps(\cdot, 0) \|_{L^2(\Omega)} = \| F(\cdot, 0) + \nabla \cdot (\ahom \nabla f ) \|_{L^2(\Omega)}$.
This yields the second energy estimate
\begin{eqnarray}
\label{data-est-2}
\nonumber\lefteqn{\| \partial_{tt} \hatueps \|_{L^{\infty}(0,T;L^2(\Omega))}
+ \| \partial_{t} \hatueps \|_{L^{\infty}(0,T;H^1(\Omega))}}\\
\nonumber&\lesssim_T& \| \partial_{t} F \|_{{L^2}(0,T;L^2(\Omega))}  + \| \nabla g \|_{L^2(\Omega)} +\| \partial_{tt} \hatueps(\cdot, 0) \|_{L^2(\Omega)}\\
\nonumber&=& \| \partial_{t} F \|_{{L^2}(0,T;L^2(\Omega))}  + \| \nabla g \|_{L^2(\Omega)} +\| F(\cdot, 0) + \nabla \cdot (\ahom \nabla f ) \|_{L^2(\Omega)} \\
&\lesssim_T& \| F \|_{W^{1,2}(0,T;L^2(\Omega))}  + \| \nabla g \|_{L^2(\Omega)} +\| \nabla \cdot (\ahom \nabla f ) \|_{L^2(\Omega)},
\end{eqnarray} 
where we used $\| F \|_{L^{\infty}(L^2)} \lesssim_T \| F \|_{W^{1,2}(L^2)}$.

$\\$
{\it Step 2 (error splitting).} Next, we let $\uHk \in H^2(0,T;V_H)$ denote the solution of the multiscale method (\ref{semi-discrete-lod-equation}) and we let $\hatuHk \in H^2(0,T;V_H)$ denote the solution of (\ref{semi-discrete-lod-equation}) with $f$ replaced by $f^{\eps}$. 
For simplicity, we also define
$$
\umsk:=\uHk + Q_{h,k}(\uHk) \qquad \mbox{and} \qquad \hatumsk := \hatuHk + Q_{h,k}(\hatuHk).
$$
With that, we split the total error in the following contributions
\begin{eqnarray*}
\lefteqn{\| \uhom - \uHk \|_{L^{\infty}(L^2)} \le
}\\
&\en&
\underset{=:\mbox{I}}{\underbrace{\| \uhom - \hatueps \|_{L^{\infty}(L^2)}}} +
\underset{=:\mbox{II}}{\underbrace{\| \hatumsk - \umsk \|_{L^{\infty}(L^2)}}} +
\underset{=:\mbox{III}}{\underbrace{\| \hatueps - \hatumsk  \|_{L^{\infty}(L^2)}}} +
\underset{=:\mbox{IV}}{\underbrace{\| Q_{h,k}(\uHk) \|_{L^{\infty}(L^2)}}}.
\end{eqnarray*}
{\it Step 3.} We start with estimating $I$. Let us denote $\hat{e}^{\eps}:=\ueps-\hatueps$. Exploiting the definitions of $\ueps$ and $\hatueps$ we obtain via \eqref{wave-equation-weak} that
\begin{align*}
\langle \partial_{tt} \hat{e}^{\eps} \dott, v \rangle + \left( \aeps \nabla \hat{e}^{\eps} \dott, \nabla v \right)_{L^2(\Omega)} &= 0 \qquad \mbox{for all } v \in H^1_0(\Omega).
\end{align*}
Furthermore, we have $\hat{e}^{\eps}( \cdot , 0 ) = f-f^{\eps}$ and
$\partial_t \hat{e}^{\eps}( \cdot , 0 ) = 0$. 
Testing with $v(\cdot,t)=\int_{t}^{t_0} \hat{e}^{\eps} ( \cdot, s ) \hspace{2pt} ds$ for any $t_0 \in [0,T]$ we obtain
\begin{eqnarray*}
\lefteqn{\frac{d}{dt} \langle \partial_{t} \hat{e}^{\eps} \dott, \int_{t}^{t_0} \hat{e}^{\eps} ( \cdot, s ) \hspace{2pt} ds \rangle
- \frac{1}{2} \frac{d}{dt}\left( \aeps \nabla \int_{t}^{t_0} \hat{e}^{\eps}( \cdot, s ) \hspace{2pt} ds, \nabla \int_{t}^{t_0} \hat{e}^{\eps} ( \cdot, s ) \hspace{2pt} ds \right)_{L^2(\Omega)}
}\\ 
&=& - \frac{1}{2} \frac{d}{dt} 
\left( \hat{e}^{\eps} \dott, \hat{e}^{\eps} \dott \right)_{L^2(\Omega)}. \hspace{250pt}
\end{eqnarray*}
Consequently by integration over the interval $[0,t_0]$ and using $\partial_{t} \hat{e}^{\eps} ( \cdot , 0)=0$ we get
\begin{eqnarray*}
\lefteqn{\| \hat{e}^{\eps}( \cdot , t_0 ) \|_{L^2(\Omega)}^2}\\
&=& \| \hat{e}^{\eps}( \cdot , 0 ) \|_{L^2(\Omega)}^2
- \beps( \int_{0}^{t_0} \hat{e}^{\eps}( \cdot, s ) \hspace{2pt} ds, \int_{0}^{t_0} \hat{e}^{\eps}( \cdot, s ) \hspace{2pt} ds)
+ 2 \langle \partial_{t} \hat{e}^{\eps} ( \cdot , 0), \int_{0}^{t_0} \hat{e}^{\eps} ( \cdot, s ) \hspace{2pt} ds \rangle\\
&\le& \| \hat{e}^{\eps}( \cdot , 0 ) \|_{L^2(\Omega)}^2 = \| f - f^{\eps} \|_{L^2(\Omega)}^2.
\end{eqnarray*}
Consequently, $I \le \| \uhom - \ueps \|_{L^{\infty}(L^2)} + \| f - f^{\eps} \|_{L^2(\Omega)} = e_{\mbox{\tiny \rm hom}}(\eps)$.

$\\$
{\it Step 4.} We can estimate term $II$ in a similar way as term $I$. Hence, proceeding as before (and exploiting definition \eqref{semi-discrete-lod-equation}) gives us
\begin{align*}
\mbox{II} = \| \hatumsk - \umsk \|_{L^{\infty}(L^2)} \le \| (\umsk - \hatumsk)( \cdot , 0 ) \|_{L^2(\Omega)}.
\end{align*}
Now we need to distinguish two cases. Case 1: If $(\umsk - \hatumsk)( \cdot , 0 )=P_{H,k}^{\ms}(f -f^{\eps})$ we can exploit the $L^2$-stability of $P_{H,k}^{\ms}$ to directly get
\begin{align*}
\mbox{II} \le \| P_{H,k}^{\ms}(f -f^{\eps}) \|_{L^2(\Omega)} \lesssim \| f -f^{\eps} \|_{L^2(\Omega)} \le e_{\mbox{\tiny \rm hom}}(\eps).
\end{align*}
Case 2: $(\umsk - \hatumsk)( \cdot , 0 )=\pi_{H,k}^{\ms}(f -f^{\eps})$. In this case, we do not have $L^2$-stability and get the following estimate
\begin{eqnarray*}
\lefteqn{\mbox{II} = \| \hatumsk - \umsk \|_{L^{\infty}(L^2)} \le \|  \pi_{H,k}^{\ms}(f -f^{\eps}) \|_{L^2(\Omega)}}\\
&\le& \|  f -f^{\eps} \|_{L^2(\Omega)} + \| f -f^{\eps}  - \pi_{H,k}^{\ms}(f -f^{\eps}) \|_{L^2(\Omega)}\\
&\le& \|  f -f^{\eps} \|_{L^2(\Omega)} + \|  f -f^{\eps}  - \pi_{h}(f -f^{\eps}) \|_{L^2(\Omega)}
+  ( H + \theta^k k^{d/2}) \| \nabla \pi_{h}( f -f^{\eps}) \|_{L^2(\Omega)}\\
&\lesssim& \|  f -f^{\eps} \|_{L^2(\Omega)} + \|  f -f^{\eps}  - \pi_{h}(f -f^{\eps}) \|_{L^2(\Omega)}
+  ( H + \theta^k k^{d/2}) \| \nabla f \|_{L^2(\Omega)},
\end{eqnarray*}
where we exploited the $H^1$-stability of $\pi_{h}$, the energy estimate $\| \nabla f^{\eps} \| \le \frac{\beta}{\alpha} \| \nabla f \|_{L^2(\Omega)}$
and used the identity \eqref{L2-estimate-for-general-v-prestage}. Consequently, in both cases we obtain 
\begin{align*}
\mbox{II} \lesssim e_{\mbox{\tiny \rm hom}}(\eps) + \hat{e}_{\mbox{\tiny \rm disc}}(h) + ( H + \theta^k k^{d/2}) \| \nabla f \|_{L^2(\Omega)}.
\end{align*}
{\it Step 5.} Next, we treat the third term $\mbox{III}= \| \hatueps - \hatumsk  \|_{L^{\infty}(L^2)}$. We start with the triangle inequality to obtain $\| \hatueps - \hatumsk  \|_{L^{\infty}(L^2)} \le \| \hatueps - \pi_{H,k}^{\ms}(\hatueps) \|_{L^{\infty}(L^2)} + \| \psi^{\pi}\|_{L^{\infty}(L^2)}$, where $\psi^{\pi}:= \hatumsk - \pi_{H,k}^{\ms}(\hatueps)$. For $\| \psi^{\pi}\|_{L^{\infty}(L^2)}$ we can use estimate (\ref{lemma-first-est-befor-tri-ineq}) that we obtained in the proof of Lemma \ref{lemma-apriori-semidiscrete}, i.e. we have
\begin{align}
\label{lemma-first-est-befor-tri-ineq-2}\| \psi^{\pi} \|_{L^{\infty}(0,T;L^2(\Omega))} \lesssim \| \psi^{\pi}(\cdot,0) \|_{L^2(\Omega)} + \| \partial_{t} \hatueps - \pi_{H,k}^{\ms}(\partial_{t} \hatueps) \|_{L^1(0,T;L^2(\Omega))}.
\end{align}
Again, we need to distinguish between $\hatumsk(\cdot,0)=\pi_{H,k}^{\ms}(f^{\eps})$ and $\hatumsk(\cdot,0)=P_{H,k}^{\ms}(f^{\eps})$. In the first case, we observe $\psi^{\pi}(\cdot,0)=0$ and the remaining estimate can be established using Lemma \ref{lemma-multiscale-projection}.
Consequently, we consider the non-trivial $L^2$-projection case, i.e. $\hatumsk(\cdot,0)=P_{H,k}^{\ms}(f^{\eps})$. In this case, we can use the triangle inequality in (\ref{lemma-first-est-befor-tri-ineq-2}) for $\psi^{\pi}=(\hatumsk - \hatueps) + (\hatueps - \pi_{H,k}^{\ms}(\hatueps))$ to obtain
\begin{eqnarray}
\label{estimate-ems}\nonumber\lefteqn{\mbox{III}=\| \hatumsk - \hatueps \|_{L^{\infty}(0,T;L^2(\Omega))}  \lesssim \underset{=:\mbox{III}_1}{\underbrace{\| (\hatumsk - \hatueps)(\cdot,0) \|_{L^2(\Omega)} +  \| (\hatueps - \pi_{H,k}^{\ms}(\hatueps))(\cdot,0) \|_{L^2(\Omega)}}}}\\
&+& \| \hatueps - \pi_{H,k}^{\ms}(\hatueps) \|_{L^{\infty}(0,T;L^2(\Omega))} + \| \partial_{t} \hatueps - \pi_{H,k}^{\ms}(\partial_{t} \hatueps) \|_{L^1(0,T;L^2(\Omega))}. \qquad \hspace{60pt}
\end{eqnarray}
The first term on the right hand side \eqref{estimate-ems} can be estimated as follows
\begin{eqnarray}
\label{estimate-initial-value-f}\nonumber\mbox{III}_1 &=& \| f^{\eps} - P_{H,k}^{\ms}(f^{\eps}) \|_{L^2(\Omega)} + \| f^{\eps} - \pi_{H,k}^{\ms}(f^{\eps})\|_{L^2(\Omega)}\\
\nonumber&\le& \inf_{v \in \Vmsk} \| f^{\eps} - v \|_{L^2(\Omega)} + \| f^{\eps} - \pi_{H,k}^{\ms}(f^{\eps}) \|_{L^2(\Omega)} \\
&\le& 2 \| f^{\eps} - \pi_{H,k}^{\ms}(f^{\eps})\|_{L^2(\Omega)}
\overset{(\ref{L2-estimate-for-general-v-prestage})}{\lesssim} \| f^{\eps} -\pi_{h}(f^{\eps}) \|_{L^2(\Omega)} + (H + \theta^k k^{d/2} ) \| f \|_{H^1(\Omega)},
\end{eqnarray}
where we again used $\| f^{\eps} \|_{H^1(\Omega)}\lesssim \| f \|_{H^1(\Omega)}$. Consequently with (\ref{equation-for-projection}), we have
\begin{eqnarray}
\label{estimate-for-term-II}
\lefteqn{\lefteqn{\mbox{III} \lesssim
\| f^{\eps} -\pi_{h}(f^{\eps}) \|_{L^2(\Omega)} + (H + \theta^k k^{d/2} ) \| f \|_{H^1(\Omega)}}} \\
\nonumber &+& \|  \hatueps - P_H(\pi_{H,k}^{\ms}(\hatueps)) \|_{L^{\infty}(L^2)} + \| \partial_{t} (\hatueps - P_H(\pi_{H,k}^{\ms}(\hatueps))) \|_{L^1(L^2)}\\
\nonumber &+& \|(Q_{H,k}^{\ms} \circ P_H \circ \pi_{H,k}^{\ms})(\hatueps) \|_{L^{\infty}(L^2)} + \| (Q_{H,k}^{\ms} \circ P_H \circ \pi_{H,k}^{\ms})(\partial_{t} \hatueps ) \|_{L^1(L^2)}.
\end{eqnarray}
The terms $\|  \hatueps - P_H(\pi_{H,k}^{\ms}(\hatueps)) \|_{L^{\infty}(L^2)}$ and $\| \partial_{t} (\hatueps - P_H(\pi_{H,k}^{\ms}(\hatueps))) \|_{L^1(L^2)}$ can be estimated with Lemma \ref{lemma-multiscale-projection}, inequality (\ref{projection-error-estimate-1}).
For the last two terms involving $Q_{H,k}(\cdot)$, we can subtract the Cl\'ement-type interpolation $I_H(Q_{H,k}(v))=0$ and use the interpolation error estimate (\ref{lemma-properties-clement-op-est}). This gives us $\mathcal{O}(H)$-terms. Then, we can then use the $H^1$-stability estimates in (\ref{stability-estimates-for-various-ops}) to obtain
\begin{eqnarray*}
\lefteqn{\mbox{III} \lesssim \| f^{\eps} -\pi_{h}(f^{\eps}) \|_{L^2(\Omega)} +  \| \hatueps - \pi_h(\hatueps) \|_{L^{\infty}(0,T;L^2(\Omega))} +  \| \partial_t \hatueps - \pi_h(\partial_t \hatueps) \|_{L^1(0,T;L^2(\Omega))}}\\
\nonumber
&+& ( H + \theta^k k^{d/2}) \left(  \| f \|_{H^1(\Omega)} +  \|  \hatueps \|_{L^{\infty}(0,T;H^1(\Omega))} + \| \partial_t \hatueps \|_{L^1(0,T;H^1(\Omega))} \right).\hspace{80pt}
\end{eqnarray*}
Using the energy estimates \eqref{data-est-1} and \eqref{data-est-2} and the definition of 
$\hat{e}_{\mbox{\tiny \rm disc}}(h)$, we get
\begin{align*}
\mbox{III} \lesssim_T  \hat{e}_{\mbox{\tiny \rm disc}}(h) + ( H + \theta^k k^{d/2}) \left(  \| F \|_{W^{1,2}(L^2)}  + \| f \|_{H^1(\Omega)} +  \| g \|_{H^1(\Omega)} +\| \nabla \cdot (\ahom \nabla f ) \|_{L^2(\Omega)} \right).
\end{align*}

$\\$
{\it Step 6.} The term $\mbox{IV}=\| Q_{h,k}(\uHk) \|_{L^{\infty}(L^2)}$ can be also estimated in the same way, using (\ref{lemma-properties-clement-op-est}) and (\ref{stability-estimates-for-various-ops}). Finally, the energy estimate
$\| \umsk \|_{L^{\infty}(H^1)} \lesssim_T \|F\|_{L^2(L^2)} + \|f\|_{H^1(\Omega)} + \| g\|_{L^2(\Omega)}$ is required to bound $\mbox{IV}$ independent of $H$ and $h$. We obtain the same type of estimate as for $\mbox{III}$. Combining the estimates for $\mbox{I-IV}$ finishes the proof.
\end{proof}

\section{Numerical experiments}
\label{section-numerics}

In this section we present the results for four different model problems. The first model problem is taken from \cite{OwZ08} and involves a microstructure without scale separation, which however can be described by a smooth coefficient. In the second model problem we abandon the smoothness and consider a problem which involves a highly heterogenous discontinuous coefficient. The third model problem is also inspired by a problem presented in \cite{OwZ08}. Here, we add an additional conductivity channel to the heterogenous structure of model problem 2, which results in a high contrast of order $10^4$. Finally, in the last experiment, we investigate
the behavior of the method for smooth, but not well-prepared initial values.

In all computations, we fix the considered time interval to be $[0,T]:=[0,1]$ and the time step size to be $\triangle t:=0.05$. In order to compute the errors for the obtained multiscale approximations, we use a discrete reference solution $u_{h,\triangle t}$ as an approximation to the exact solution of problem (\ref{wave-equation-strong}). This reference solution is determined with the Crank-Nicolson scheme for the time discretization (using equidistant time steps with time step size $\triangle t=0.05$) and a Finite Element method on the fine mesh $\T_h$ for the space discretization. We use a linear interpolation between the solutions obtained for each time step. Hence, $\partial_t u_{h,\triangle t}$ is well defined on each time interval $[t^n,t^{n+1}]$. By $\uHtk$ we denote the multiscale approximation defined according to (\ref{fully-discrete-lod-approximation}).

In this section, we use the following notation for the errors:
\begin{align}
\nonumber e^{0,n} &:= \uHtk(\cdot,t^n) - u_{h,\triangle t}(\cdot,t^n)\\
\label{errors-definition} e^{\ms,n} &:= (\uHtk+Q_{h,k}(\uHtk))(\cdot,t^n)  - u_{h,\triangle t}(\cdot,t^n),\\
\nonumber \partial_t e^{\ms,n} &:= \lim_{t \nearrow t^n} \partial_t(\uHtk+Q_{h,k}(\uHtk))(\cdot,t)  - \partial_t u_{h,\triangle t}(\cdot,t).
\end{align}
By $\| \cdot \|_{L^2(\Omega)}^{\mbox{\tiny rel}}$ (respectively $\| \cdot \|_{H^1(\Omega)}^{\mbox{\tiny rel}}$) we denote the relative error norms, i.e. the absolute errors divided by the associated norm of the reference solution $u_{h,\triangle t}$. Furthermore, for an error $\| e_H \|$ on a coarse grid $\T_H$ and an error $\| e_{H/2} \|$ on a coarse grid $\T_{H/2}$, the EOC (experimental order of convergence) is given by $\mbox{EOC}_{H}:=\log_2(\| e_H \| / \| e_{H/2} \|)/\log_2(2)$.

\subsection{Model problem 1}

The first model problem is extracted from \cite{OwZ08}. As pointed out in \cite{OwZ08}, a sufficiently accurate reference solution $u_{h,\triangle t}$ is obtained for a uniform fine grid with resolution $h=2^{-7}$. Hence, we fix $\T_h$ to be a uniformly refined triangulation of $\Omega$ with $66.049$ DOFs.

\begin{problem}
Let $\Omega:= ]-1,1[^2$ and $T:=1$. Find $\ueps \in L^{\infty}(0,T;H^1_0(\Omega))$ such that
\begin{align}
\nonumber\label{eq:model} \partial_{tt} \ueps(x,t) - \nabla \cdot \left( \aeps(x) \nabla \ueps(x,t) \right) &= F(x) \qquad \mbox{in } \Omega \times(0,T], \\
\ueps(x,t) &= 0 \hspace{35pt} \mbox{on } \partial \Omega \times [0,T], \\
\nonumber\ueps(x,0) = 0 \quad \mbox{and} \quad \partial_t \ueps(x,0) &= 0 \hspace{39pt} \mbox{in } \Omega,
\end{align}
where $F$ is a Gaussian source term given by $F(x_1,x_2)=(2 \pi \sigma^2)^{-1/2} e^{-(x_1^2 + (x_2-0.15)^2)/(2 \sigma^2)}$ for $\sigma=0.05$ and
\begin{align}
\label{diff-coefficient-1}\aeps(x_1,x_2)&:=\frac{1}{6} \left( 1 + \sin(4 x_1^2 x_2^2) + \frac{1.1 + \sin(2 \pi x_1 / \eps_1)}{1.1 + \sin(2 \pi x_2 / \eps_1)} + \frac{1.1 + \sin(2 \pi x_1 / \eps_2)}{1.1 + \cos(2 \pi x_2 / \eps_2)} \right.\\
\nonumber&\left. + \frac{1.1 + \cos(2 \pi x_1 / \eps_3)}{1.1 + \sin(2 \pi x_2 / \eps_3)} +  \frac{1.1 + \sin(2 \pi x_1 / \eps_4)}{1.1 + \cos(2 \pi x_2 / \eps_4)} + \frac{1.1 + \cos(2 \pi x_1 / \eps_5)}{1.1 + \sin(2 \pi x_2 / \eps_5)} \right).
\end{align}
with $\eps_1=1/5$, $\eps_2=1/13$, $\eps_3=1/17$, $\eps_4=1/31$ and $\eps_5=1/65$. The coefficient $\aeps$ is plotted in Figure \ref{diffusion-plot-1}, together with the reference solution $u_{h,\triangle t}$ for $t=1$.
\end{problem}
Note that the Gaussian source term will become singular for $\sigma \rightarrow 0$. Hence it influences the regularity of the solution and we expect the multiscale approximation to be less accurate than for a more regular source term. In particular, $F$ has already a very large $H^1$-norm, which is why we cannot expect to see the third order convergence O$(H^3)$ in (\ref{main-result-est-3-fully-discrete}), unless $H/\|F\|_{H^1(\Omega)} \ll 1$.

{\small
\begin{figure}[h!]
\centering
\includegraphics[scale=0.22]{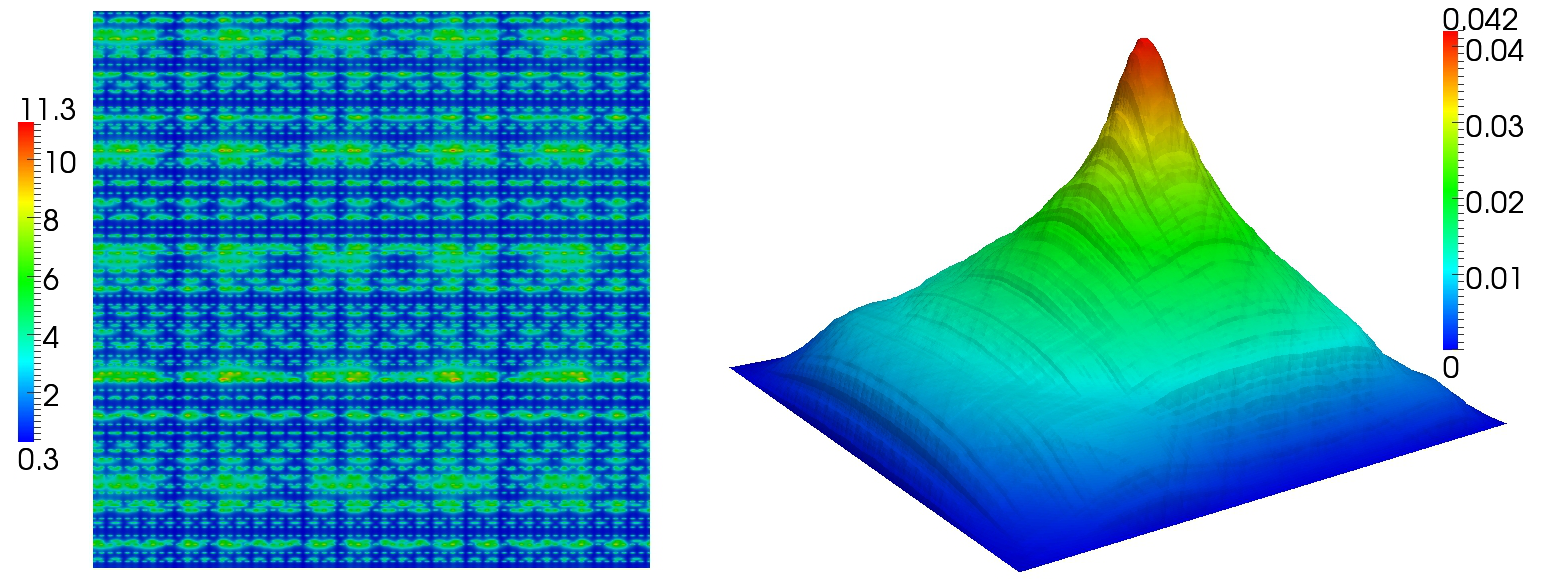}
\caption{\it Model Problem 1. Left Picture: Plot of the coefficient $\aeps$ given by (\ref{diff-coefficient-1}). Right Picture: reference solution $u_{h,\triangle t}$ at $t=1$ for $h=2^{-7}$.}
\label{diffusion-plot-1}
\end{figure}
\begin{figure}[h!]
\centering
\includegraphics[scale=0.22]{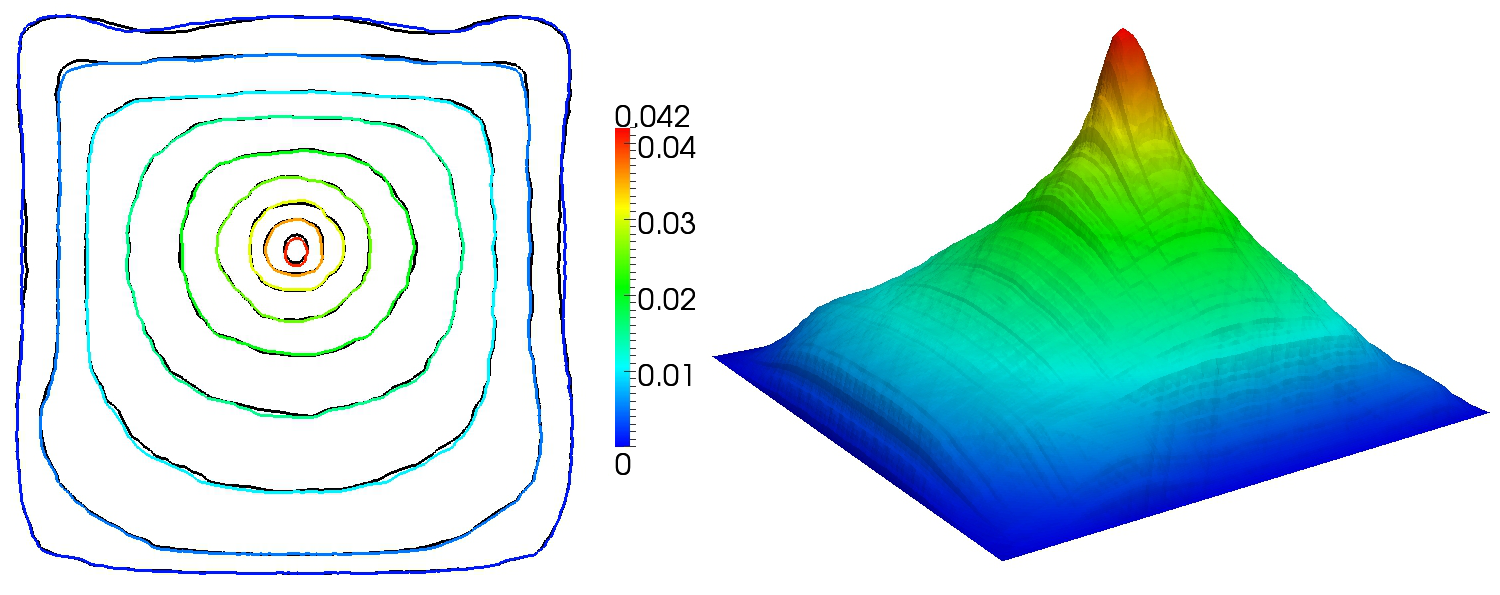}
\caption{\it Model Problem 1, results for $t^n=1$. Left Picture: Comparison of the isolines of the reference solution $u_{h,\triangle t}$ for $h=2^{-7}$ (black isolines) with the multiscale approximation $\uHtk+Q_{h,k}(\uHtk)$ for $(H,h,k)=(2^{-3},2^{-7},2)$ (colored isolines). Right Picture: Plot of the multiscale approximation $\uHtk+Q_{h,k}(\uHtk)$ for $(H,h,k)=(2^{-3},2^{-7},2)$.}
\label{model-problem-1-isolines}
\end{figure}
}
{\small
\begin{table}[h!]
\caption{\it Model Problem 1, results for $t^n=1$. The table depicts relative $L^2$- and $H^1$-errors for the obtained multiscale approximations with respect to the reference solution. The errors are defined in (\ref{errors-definition}).}
\label{table-layers-results-1}
\begin{center}
\begin{tabular}{|c|c|c|c|c|c|c|c|}
\hline $H$      & $k$
& $\| e^{0,n} \|_{L^2(\Omega)}^{\mbox{\tiny rel}}$
& $\| e^{\ms,n} \|_{L^2(\Omega)}^{\mbox{\tiny rel}}$
& $\| e^{\ms,n} \|_{H^1(\Omega)}^{\mbox{\tiny rel}}$
& $\| \partial_t e^{\ms,n} \|_{L^2(\Omega)}^{\mbox{\tiny rel}}$
& $\| \partial_t e^{\ms,n} \|_{H^1(\Omega)}^{\mbox{\tiny rel}}$\\
\hline
\hline $2^{-1}$ & 1   & 0.1448   & 0.1341 & 0.4532 & 0.8718 & 0.9957 \\
\hline $2^{-1}$ & 2   & 0.1394   & 0.1334 & 0.4627 & 0.8312 & 0.9822 \\
\hline
\hline $2^{-2}$ & 1   & 0.0780 &  0.0688 & 0.3517  & 0.6464 & 0.9424 \\
\hline $2^{-2}$ & 2   & 0.0687 &  0.0521 & 0.2919  & 0.5439 & 0.8949 \\
\hline $2^{-2}$ & 3   & 0.0675 &  0.0499 & 0.2835  & 0.5362 & 0.8929 \\
\hline
\hline $2^{-3}$ & 1   & 0.0368 &  0.0328 & 0.2279  & 0.5824 & 1.1262 \\
\hline $2^{-3}$ & 2   & 0.0242 &  0.0130 & 0.1212  & 0.3285 & 0.7769 \\
\hline $2^{-3}$ & 3   & 0.0234 &  0.0105 & 0.1036  & 0.2846 & 0.6998 \\
\hline
\end{tabular}\end{center}
\end{table}
}
\begin{table}[h!]
\caption{\it Model Problem 1, results for $t^n=1$. Overview on the EOCs associated with errors from Table \ref{table-layers-results-1}. We couple $k$ and $H$ by $k=k(H):=\lfloor |\ln(H)| + 1 \rfloor$. For each of the errors $\|e_H\|$ below (for $H=2^{-i}$), we define the average EOC by EOC$:= \frac{1}{2} \sum_{i=1}^2 \log_2(\| e_{2^{-i}} \| / \| e_{2^{-(i+1)}} \|)/\log_2(2)$.}
\label{table-EOC-results-1}
\begin{center}
\begin{tabular}{|c|c|c|c|c|c|c|c|}
\hline $H$      & $k(H)$
& $\| e^{0,n} \|_{L^2(\Omega)}^{\mbox{\tiny rel}}$
& $\| e^{\ms,n} \|_{L^2(\Omega)}^{\mbox{\tiny rel}}$
& $\| e^{\ms,n} \|_{H^1(\Omega)}^{\mbox{\tiny rel}}$
& $\| \partial_t e^{\ms,n} \|_{L^2(\Omega)}^{\mbox{\tiny rel}}$
& $\| \partial_t e^{\ms,n} \|_{H^1(\Omega)}^{\mbox{\tiny rel}}$\\
\hline
\hline $2^{-1}$ & 1   & 0.1448   & 0.1341 & 0.4532 & 0.8718 & 0.9957 \\
\hline $2^{-2}$ & 2   & 0.0687  &  0.0521 & 0.2919  & 0.5439 & 0.8949 \\
\hline $2^{-3}$ & 3   & 0.0234 &  0.0105 & 0.1036  & 0.2846 & 0.6998 \\
\hline
\hline \multicolumn{2}{|c|}{EOC}  & 1.31 & 1.84  & 1.06  & 0.81 & 0.25 \\
\hline
\end{tabular}\end{center}
\end{table}

In Table \ref{table-layers-results-1} the relative errors are depicted for various combinations of $H$ and $k$ (recall that $k$ denotes the truncation parameter defined in (\ref{def-patch-U-k})). The errors are qualitatively comparable to the errors obtained in \cite{OwZ08} for similar computations. Furthermore, we observe that the error evolution is consistent with the theoretically predicated rates. EOCs are given in Table \ref{table-EOC-results-1}. 

For $k\approx |\ln(H)|+1$, we observe roughly a convergence rate of $1.3$ in $H$ for the $L^2$-error of the numerically homogenized solution $\uHtk$. Adding the corresponding corrector $Q_{h,k}(\uHtk)$, the rate is close to $2$ in average. In Figure \ref{model-problem-1-isolines}, a visual comparison between the reference solution and the multiscale approximation is shown. We observe that for $(H,h,k)=(2^{-3},2^{-7},2)$, the solution $\uHtk+Q_{h,k}(\uHtk)$ looks the same as the reference solution $u_{h,\triangle t}$ depicted in (\ref{diffusion-plot-1}). This is also stressed by the comparison of isolines in Figure \ref{model-problem-1-isolines}.

\subsection{Model problem 2}

\begin{figure}
\centering
\includegraphics[scale=0.22]{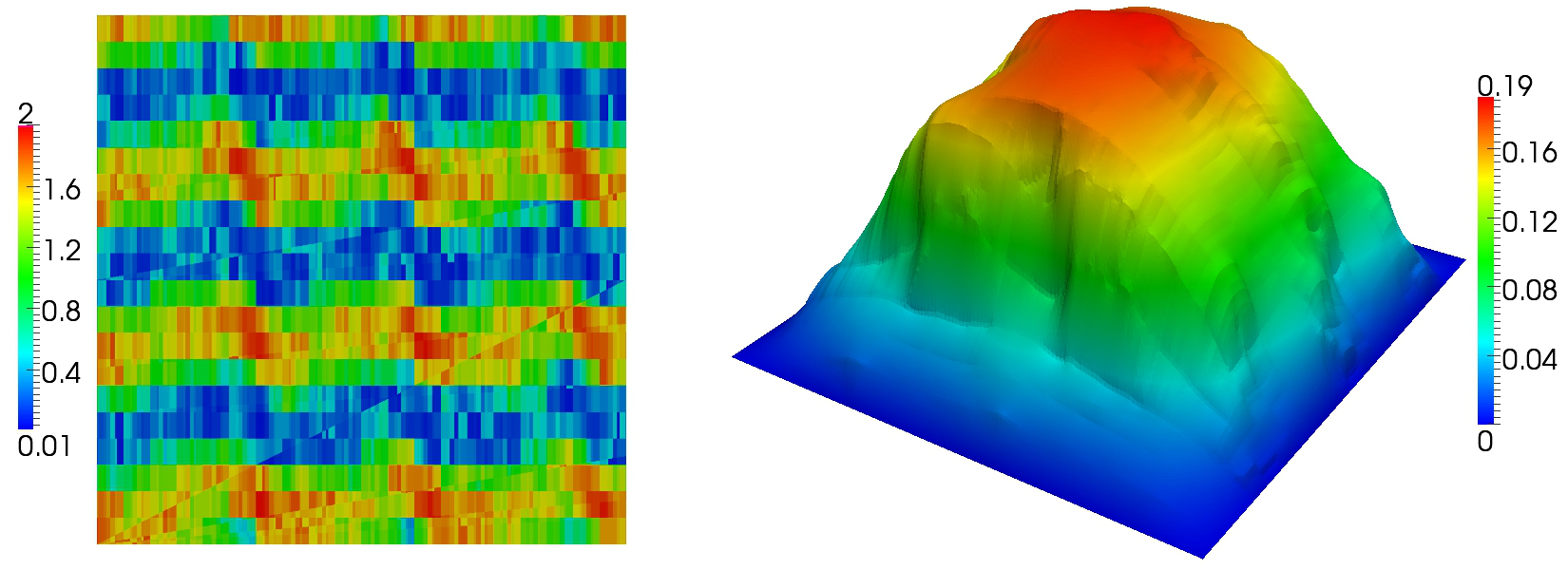}
\caption{\it Model Problem 2, plots for $t^n=1$. Left Picture: Plot of the coefficient $\aeps$ given by (\ref{diff-coefficient-2}). Right Picture: reference solution $u_{h,\triangle t}$ at $t=1$ for $h=2^{-8}$.}
\label{diffusion-plot-2}
\end{figure}

In Model Problem 2 we investigate the influence of a discontinuous coefficient $\aeps$ in our multiscale method. According to the theoretical results, it should not influence the convergence rates. The fine grid $\T_h$ is a uniformly refined triangulation with resolution $h=2^{-8}$.

\begin{problem}
Let $\Omega:= ]0,1[^2$ and $T:=1$. Find $\ueps \in L^{\infty}(0,T;H^1_0(\Omega))$ such that
\begin{align}
\nonumber\label{eq:model} \partial_{tt} \ueps(x,t) - \nabla \cdot \left( \aeps(x) \nabla \ueps(x,t) \right) &= 1 \qquad \hspace{17pt} \mbox{in } \Omega \times(0,T], \\
\ueps(x,t) &= 0 \hspace{37pt} \mbox{on } \partial \Omega \times [0,T], \\
\nonumber\ueps(x,0) = 0 \quad \mbox{and} \quad \partial_t \ueps(x,0) &= 0 \hspace{38pt} \mbox{in } \Omega.
\end{align}
Here, we have
\begin{align}
\label{diff-coefficient-2}\aeps(x)&:= (h \circ c_\varepsilon)(x)
\qquad \text{with} \enspace h(t):=\begin{cases}
t^4 &\text{for} \enspace \frac12 < t < 1 \\ 
t^{\frac{3}{2}} &\text{for} \enspace 1 < t < \frac{3}{2}  \\ 
t &\text{else}
\end{cases}
\end{align}
and where
\begin{displaymath}
c_\varepsilon(x_1,x_2):=1 + \frac{1}{10} \sum_{j=0}^4 \sum_{i=0}^{j} \left( \frac{2}{j+1} \cos \left( \bigl\lfloor i x_2 - \tfrac{x_1}{1+i} \bigr\rfloor + \left\lfloor \tfrac{i x_1}\varepsilon \right\rfloor + \left\lfloor \tfrac{ x_2}\varepsilon \right\rfloor \right) \right).
\end{displaymath}
The coefficient $\aeps$ is plotted in Figure \ref{diffusion-plot-2} together with the reference solution on $\T_h$ for $t=1$.
\end{problem}

\begin{table}[h!]
\caption{\it Model Problem 2. Overview on relative $L^2$- and $H^1$-errors for Model Problem 2 for $t^n=1$. The errors are defined in (\ref{errors-definition}).}
\label{table-layers-results-2}
\begin{center}
\begin{tabular}{|c|c|c|c|c|c|c|c|}
\hline $H$      & $k$
& $\| e^{0,n} \|_{L^2(\Omega)}^{\mbox{\tiny rel}}$
& $\| e^{\ms,n} \|_{L^2(\Omega)}^{\mbox{\tiny rel}}$
& $\| e^{\ms,n} \|_{H^1(\Omega)}^{\mbox{\tiny rel}}$
& $\| \partial_t e^{\ms,n} \|_{L^2(\Omega)}^{\mbox{\tiny rel}}$
& $\| \partial_t e^{\ms,n} \|_{H^1(\Omega)}^{\mbox{\tiny rel}}$\\
\hline
\hline $2^{-2}$ & 1   &  0.1299 & 0.0613  & 0.1802 & 0.1762 & 0.6615 \\
\hline $2^{-2}$ & 2   &  0.1223 & 0.0245  & 0.0800 & 0.1298 & 0.6323 \\
\hline
\hline $2^{-3}$ & 1   &  0.0914 & 0.0616  & 0.1926 & 0.2194 & 0.7255 \\
\hline $2^{-3}$ & 2   &  0.0753 & 0.0191  & 0.0841 & 0.1049 & 0.5902 \\
\hline $2^{-3}$ & 3   &  0.0741 & 0.0085  & 0.0563 & 0.0870 & 0.5688 \\
\hline
\hline $2^{-4}$ & 1   & 0.0327 & 0.0243   & 0.1401 & 0.1197 & 0.6710 \\
\hline $2^{-4}$ & 2   & 0.0240 & 0.0047   & 0.0505 & 0.0600 & 0.5109 \\
\hline $2^{-4}$ & 3   & 0.0239 & 0.0029   & 0.0347 & 0.0562 & 0.5004 \\
\hline
\end{tabular}\end{center}
\end{table}

\begin{table}[t]
\caption{\it Model Problem 2, results for $t^n=1$. Overview on the EOCs associated with errors from Table \ref{table-layers-results-2}. We couple $k$ and $H$ by $k=k(H):=\lfloor |\ln(H)| + 0.5 \rfloor$. The average EOCs are computed according to (\ref{eoc-def}).}
\label{table-EOC-results-2}
\begin{center}
\begin{tabular}{|c|c|c|c|c|c|c|c|}
\hline $H$      & $k(H)$
& $\| e^{0,n} \|_{L^2(\Omega)}^{\mbox{\tiny rel}}$
& $\| e^{\ms,n} \|_{L^2(\Omega)}^{\mbox{\tiny rel}}$
& $\| e^{\ms,n} \|_{H^1(\Omega)}^{\mbox{\tiny rel}}$
& $\| \partial_t e^{\ms,n} \|_{L^2(\Omega)}^{\mbox{\tiny rel}}$
& $\| \partial_t e^{\ms,n} \|_{H^1(\Omega)}^{\mbox{\tiny rel}}$\\
\hline
\hline $2^{-2}$ & 1   & 0.1299  & 0.0613  & 0.1802 & 0.1762 & 0.6615 \\
\hline $2^{-3}$ & 2   & 0.0753  & 0.0191  & 0.0841 & 0.1049 & 0.5902 \\
\hline $2^{-4}$ & 3   & 0.0239  & 0.0029  & 0.0347 & 0.0562 & 0.5004 \\
\hline
\hline \multicolumn{2}{|c|}{EOC}  & 1.22 & 2.20  & 1.19  & 0.82 & 0.20 \\
\hline
\end{tabular}\end{center}
\end{table}

In Table \ref{table-layers-results-2} we depict various relative $L^2$- and $H^1$-errors for $t^n=1$. We observe that the numerically homogenized solution $\uHtk$ already yields good $L^2$-approximation properties with respect to the fine scale reference solution. These approximation properties can still be significantly improved by adding the corrector $Q_{h,k}(\uHtk)$. We can see that the total approximation $\uHtk+Q_{h,k}(\uHtk)$ is also accurate in the norms $\| \nabla \cdot \|_{L^2(\Omega)}$ and $\| \partial_t \cdot \|_{L^2(\Omega)}$. The errors in the norm $\| \partial_t \cdot \|_{H^1(\Omega)}$ are noticeably larger than the errors in all other norms. Hence, the multiscale solution does not necessarily yield a highly accurate $H^1(0,T,H^1(\Omega))$ approximation, even though the discrepancy is still tolerable.

All errors for Model Problem 2 are of the same order as the errors for Model Problem 1 depicted in Table \ref{table-layers-results-1}. In particular, we do not see any error deterioration caused by the discontinuity of the coefficient $\aeps$. The method behaves nicely in both cases. This is stressed by the experimental orders of convergence (EOCs) shown in Table \ref{table-EOC-results-2}. For the EOCs in Table \ref{table-EOC-results-2}, we use the average 
\begin{align}
\label{eoc-def} \mbox{EOC} := (\mbox{EOC}_{2^{-2}} + \mbox{EOC}_{2^{-3}})/2.
\end{align}
Motivated by Theorem \ref{apriori-fullydiscrete}, we couple the coarse mesh $H$ and the truncation parameter $k$ to be the closest integer to $|\ln(H)|$, i.e. we pick $k=k(H):=\lfloor |\ln(H)| + 0.5 \rfloor$ for the computation of the EOCs. This gives us $k=1$ for $H=2^{-2}$, $k=2$ for $H=2^{-3}$ and $k=3$ for $H=2^{-4}$. The corresponding results are stated in Table \ref{table-EOC-results-2}. We observe a close to linear convergence for the $L^2$-error for the numerically homogenized solution $\uHtk$ (as predicted by the theory). Adding the corrector $Q_{h,k}(\uHtk)$, the convergence rate increases to $2.2$ (slightly worse than the optimal rate of $3$). For the $H^1$-error, we observe linear convergence. The convergence rate of the $L^2$-error for time derivatives is slightly below linear convergence, but still very satisfying. Only the convergence rate for the $H^1$-error for the time derivatives is close to stagnation. 

Note that the deviation of these rates from the perfect rates comes from that fact that we do not know the generic constant $C_{\theta}$ in Theorem \ref{apriori-fullydiscrete}. We picked $k=|\ln(H)|$, instead of $k=C_{\theta} |\ln(H)|$. Still we observe that approximating $C_{\theta}$ by $1$ yields highly accurate results that are close to the optimal rates. Practically, this justifies the use of small localization patches $U_k(K)$.

\subsection{Model problem 3}

\begin{figure}
\centering
\includegraphics[scale=0.22]{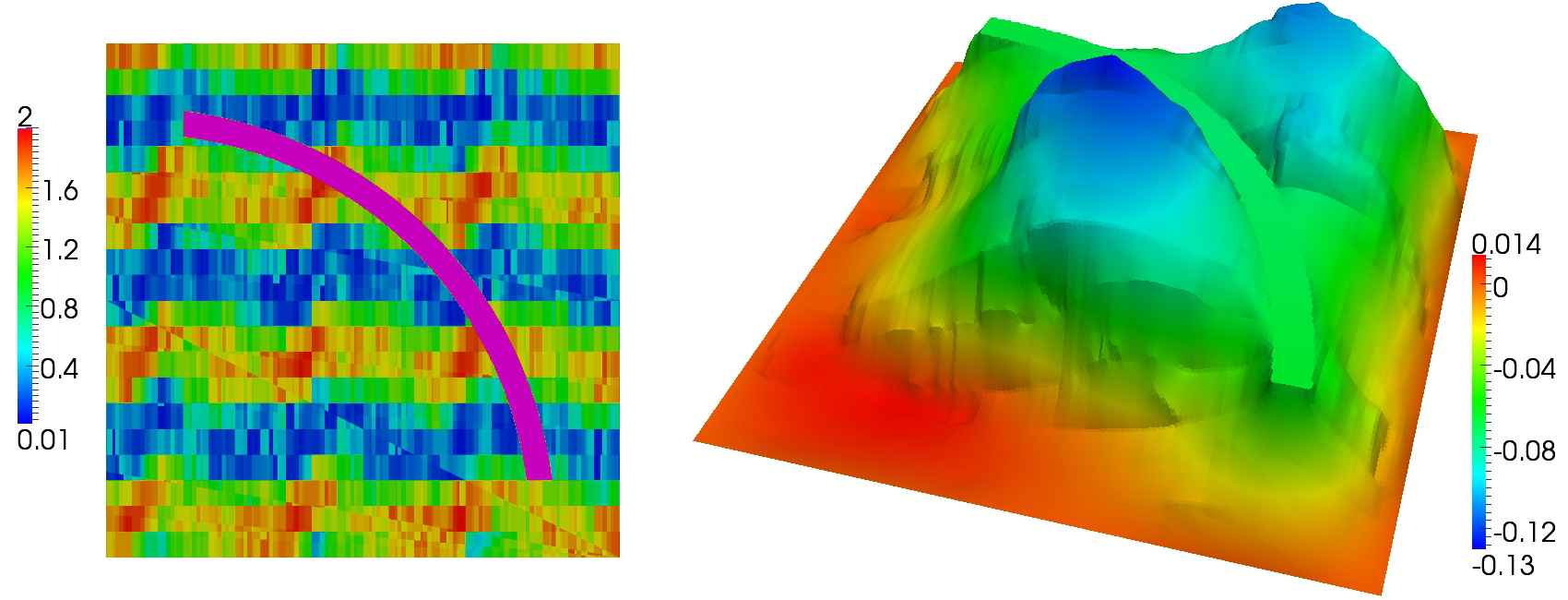}
\caption{\it Model Problem 3. Left Picture: Plot of the coefficient $\aeps$. The basis structure of $\aeps$ is given by (\ref{diff-coefficient-2}), but this structure is perturbed by an arc-like conductivity channel (pink). In this arc $\aeps$ takes the value $100$. Right Picture: reference solution $u_{h,\triangle t}$ at $t=1$ for $h=2^{-8}$.}
\label{diffusion-plot-3}
\end{figure}

\begin{problem}
Let $\Omega:= ]0,1[^2$ and $T:=1$. Find $\ueps \in L^{\infty}(0,T;H^1_0(\Omega))$ such that
\begin{align}
\nonumber\label{eq:model} \partial_{tt} \ueps(x,t) - \nabla \cdot \left( \aeps(x) \nabla \ueps(x,t) \right) &= F(x,t) \qquad \hspace{17pt} \mbox{in } \Omega \times(0,T], \\
\ueps(x,t) &= 0 \hspace{66pt} \mbox{on } \partial \Omega \times [0,T], \\
\nonumber\ueps(x,0) = 0 \quad \mbox{and} \quad \partial_t \ueps(x,0) &= 0 \hspace{66pt} \mbox{in } \Omega.
\end{align}
Here, we have $F(x_1,x_2,t) = \sin( 2.4 x_1 - 1.8 x_2 + 2 \pi t )$ and $\aeps$ is given by equation (\ref{diff-coefficient-2}) but additionally it is disturbed by a high conductivity channel of thickness 0.05. The precise structure of $\aeps$ is depicted in Figure \ref{diffusion-plot-3}, together with the reference solution on $\T_h$ for $t=1$.
\end{problem}

\begin{table}[t]
\caption{\it Model Problem 3, results for $t^n=1$. Overview on relative $L^2$- and $H^1$-errors defined as in (\ref{errors-definition}).}
\label{table-layers-results-3}
\begin{center}
\begin{tabular}{|c|c|c|c|c|c|c|c|}
\hline $H$      & $k$
& $\| e^{0,n} \|_{L^2(\Omega)}^{\mbox{\tiny rel}}$
& $\| e^{\ms,n} \|_{L^2(\Omega)}^{\mbox{\tiny rel}}$
& $\| e^{\ms,n} \|_{H^1(\Omega)}^{\mbox{\tiny rel}}$
& $\| \partial_t e^{\ms,n} \|_{L^2(\Omega)}^{\mbox{\tiny rel}}$
& $\| \partial_t e^{\ms,n} \|_{H^1(\Omega)}^{\mbox{\tiny rel}}$\\
\hline
\hline $2^{-2}$ & 1   & 0.2468  & 0.1564  & 0.3321 & 0.2066 & 0.4486 \\
\hline $2^{-2}$ & 2   & 0.2270  & 0.0782  & 0.1992 & 0.1168 & 0.3269 \\
\hline
\hline $2^{-3}$ & 1   & 0.1451  & 0.1046  & 0.3305 & 0.1639 & 0.4588 \\
\hline $2^{-3}$ & 2   & 0.1184  & 0.0329  & 0.1535 & 0.0607 & 0.2724 \\
\hline $2^{-3}$ & 3   & 0.1174  & 0.0202  & 0.1024 & 0.0468 & 0.2333 \\
\hline
\hline $2^{-4}$ & 1   &  0.0550 & 0.0433  & 0.2186 & 0.0667 & 0.3349 \\
\hline $2^{-4}$ & 2   &  0.0390 & 0.0095  & 0.0803 & 0.0250 & 0.1896 \\
\hline $2^{-4}$ & 3   &  0.0385 & 0.0046  & 0.0464 & 0.0198 & 0.1758 \\
\hline
\end{tabular}\end{center}
\end{table}

\begin{table}[t]
\caption{\it Model Problem 3, results for $t^n=1$. Overview on the EOCs associated with errors from Table \ref{table-layers-results-3}. We couple $k$ and $H$ by $k=k(H):=\lfloor |\ln(H)| + 0.5 \rfloor$. The average EOCs are computed according to (\ref{eoc-def}).}
\label{table-EOC-results-3}
\begin{center}
\begin{tabular}{|c|c|c|c|c|c|c|c|}
\hline $H$      & $k(H)$
& $\| e^{0,n} \|_{L^2(\Omega)}^{\mbox{\tiny rel}}$
& $\| e^{\ms,n} \|_{L^2(\Omega)}^{\mbox{\tiny rel}}$
& $\| e^{\ms,n} \|_{H^1(\Omega)}^{\mbox{\tiny rel}}$
& $\| \partial_t e^{\ms,n} \|_{L^2(\Omega)}^{\mbox{\tiny rel}}$
& $\| \partial_t e^{\ms,n} \|_{H^1(\Omega)}^{\mbox{\tiny rel}}$\\
\hline
\hline $2^{-2}$ & 1   & 0.2468  & 0.1564  & 0.3321 & 0.2066 & 0.4486 \\
\hline $2^{-3}$ & 2   & 0.1184  & 0.0329  & 0.1535 & 0.0607 & 0.2724 \\
\hline $2^{-4}$ & 3   & 0.0385 & 0.0046  & 0.0464  & 0.0198 & 0.1758 \\
\hline
\hline \multicolumn{2}{|c|}{EOC}  & 1.34 & 2.54  & 1.42  & 1.69 & 0.68 \\
\hline
\end{tabular}\end{center}
\end{table}

\begin{figure}
\centering
\includegraphics[scale=0.22]{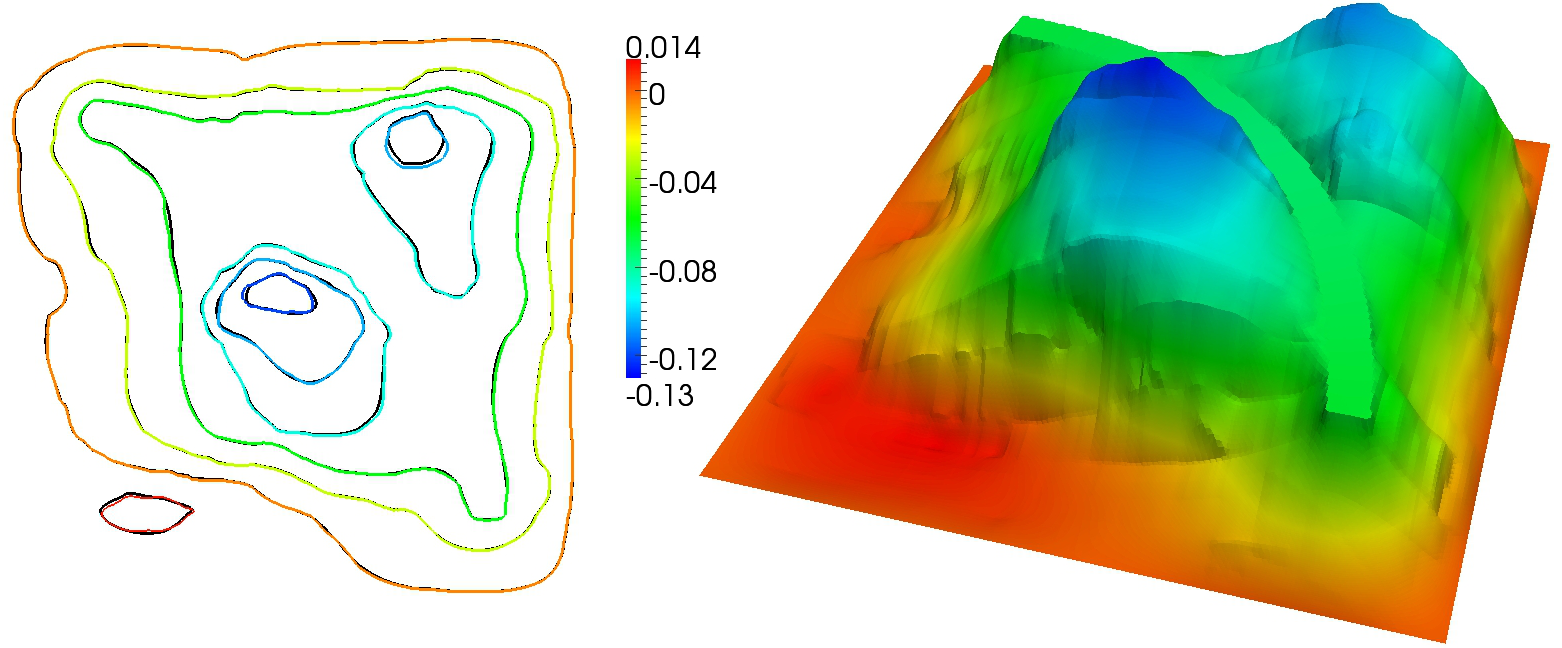}
\caption{\it Model Problem 3, results at $t=1$. Left Picture: Comparison of the isolines of the reference solution $u_{h,\triangle t}$ for $h=2^{-8}$ (black isolines) with the multiscale approximation $\uHtk+Q_{h,k}(\uHtk)$ for $(H,h,k)=(2^{-4},2^{-8},2)$ (colored isolines). Right Picture: Plot of the multiscale approximation $\uHtk+Q_{h,k}(\uHtk)$ for $(H,h,k)=(2^{-4},2^{-8},2)$.}
\label{model-problem-3-isolines}
\end{figure}

This model problem is inspired by a model problem in \cite{OwZ08}. The source term as in \cite{OwZ08} is given by $F(x_1,x_2,t) = \sin( 2.4 x_1 - 1.8 x_2 + 2 \pi t )$ and 
as in \cite{OwZ08} $\aeps$ contains a conductivity channel that perturbs the original structure. As we could not access the data for the channel given in this reference we model a new one in this paper.
This model problem is set to investigate the approximation quality of our multiscale approximations for problems with channels (which do not have to be resolved by the coarse grid) and a high contrast $\beta/\alpha \approx 10^4$. As in the previous model problem, we chose $\T_h$ as a uniformly refined triangulation with resolution $h=2^{-8}$.

We see in Table \ref{table-layers-results-3} that the additional channel in the problem does not deteriorate the convergence rates compared to Model Problem 2.
Again, close to optimal convergence rates are obtained for the choice $k=k(H):=\lfloor |\ln(H)| + 0.5 \rfloor$ (which slightly underestimates the optimal truncation parameter $k$). The corresponding results are given in Table \ref{table-layers-results-3}. The method yields accurate results even in the case of a conductivity channel and despite that the coarse grid does not resolve the channel. Furthermore the high contrast of order $10^4$ does not significantly influence the size of the optimal truncation parameter $k$. In the model problem we can still work with small localization patches $U_k(K)$, independent of the conductivity channel. This is further stressed by Figure \ref{model-problem-3-isolines} which depicts the multiscale approximation (i.e. $\uHtk+Q_{h,k}(\uHtk)$) for the case $(H,k)=(2^{-4},2)$. The solution looks almost identical to the reference solution $u_{h,\triangle t}$ and the corresponding isolines match almost perfectly.

\subsection{Model problem 4}

In this model problem we investigate the performance of the multiscale method in the case of smooth, but not well-prepared initial values. In the light of Theorem \ref{apriori-lod-homogenization}, we can only expect a linear convergence rate with respect to $H$ and in the $L^{\infty}(L^2)$-norm and we cannot expect any rates in the $W^{1,\infty}(L^2)$- and $L^{\infty}(L^2)$-norms, since Theorem \ref{apriori-semidiscrete} is not valid under general assumptions. As a model problem, we consider the setting of Model Problem 2, but with different source and initial values. As in the previous model problem, we chose $\T_h$ as a uniformly refined triangulation with resolution $h=2^{-8}$. We consider the following problem.

\begin{problem}
Let $\Omega:= ]0,1[^2$ and $T:=1$. Find $\ueps \in L^{\infty}(0,T;H^1_0(\Omega))$ such that
\begin{align}
\nonumber\label{eq:model} \partial_{tt} \ueps(x,t) - \nabla \cdot \left( \aeps(x) \nabla \ueps(x,t) \right) &= F(x,t) \qquad \hspace{17pt} \mbox{in } \Omega \times(0,T], \\
\ueps(x,t) &= 0 \hspace{66pt} \mbox{on } \partial \Omega \times [0,T], \\
\nonumber\ueps(x,0) = f \quad \mbox{and} \quad \partial_t \ueps(x,0) &= g \hspace{66pt} \mbox{in } \Omega.
\end{align}
Here, we have $F(x_1,x_2) = \sin( 2 \pi x_1 ) \sin( 2 \pi x_2 )$, $f(x_1,x_2)=x_1 (1-x_1)x_2 (1-x_2)$, $g(x_1,x_2)=\sin( 2 \pi x_1 ) x_2 (1 - x_2)$ and $\aeps$ is given by equation (\ref{diff-coefficient-2}).
\end{problem}

\begin{table}[t]{
\caption{\it Model Problem 4, results for $t^n=1$. Overview on relative $L^2$- and $H^1$-errors defined as in (\ref{errors-definition}).}
\label{table-layers-results-4}
\begin{center}
\begin{tabular}{|c|c|c|c|c|c|c|c|}
\hline $H$      & $k$
& $\| e^{0,n} \|_{L^2(\Omega)}^{\mbox{\tiny rel}}$
& $\| e^{\ms,n} \|_{L^2(\Omega)}^{\mbox{\tiny rel}}$
& $\| e^{\ms,n} \|_{H^1(\Omega)}^{\mbox{\tiny rel}}$
& $\| \partial_t e^{\ms,n} \|_{L^2(\Omega)}^{\mbox{\tiny rel}}$
& $\| \partial_t e^{\ms,n} \|_{H^1(\Omega)}^{\mbox{\tiny rel}}$\\
\hline
\hline $2^{-2}$ & 1   & 0.2809  & 0.2598  & 0.6180 & 0.3522 & 0.9735 \\
\hline $2^{-2}$ & 2   & 0.1865  & 0.1229  & 0.4954 & 0.3147 & 0.9692 \\
\hline
\hline $2^{-3}$ & 1   & 0.1579  & 0.1420  & 0.5680 & 0.3243 & 0.9691 \\
\hline $2^{-3}$ & 2   & 0.1188  & 0.0894  & 0.4869 & 0.2473 & 0.9593 \\
\hline $2^{-3}$ & 3   & 0.1145  & 0.0820  & 0.4741 & 0.2372 & 0.9573 \\
\hline
\hline $2^{-4}$ & 1   & 0.0885  & 0.0857  & 0.4774 & 0.2891 & 0.9850 \\
\hline $2^{-4}$ & 2   & 0.0466  & 0.0361  & 0.3249 & 0.1925 & 0.9506 \\
\hline $2^{-4}$ & 3   & 0.0423  & 0.0289  & 0.3042 & 0.1823 & 0.9479 \\
\hline
\hline $2^{-5}$ & 1   & 0.0499  & 0.0489  & 0.3481 & 0.1849 & 0.9593 \\
\hline $2^{-5}$ & 2   & 0.0198  & 0.0149  & 0.2256 & 0.1277 & 0.9204 \\
\hline $2^{-5}$ & 3   & 0.0173  & 0.0109  & 0.2059 & 0.1229 & 0.9171 \\
\hline
\end{tabular}\end{center}}
\end{table}

\begin{table}[t]{
\caption{\it Model Problem 4, results for $t^n=1$. Overview on the EOCs associated with errors from Table \ref{table-layers-results-4}. We couple $k$ and $H$ by $k=k(H):=\lfloor |\ln(H)| + 0.5 \rfloor$. The average EOCs are computed with $\mbox{EOC} := (\mbox{EOC}_{2^{-2}} + \mbox{EOC}_{2^{-3}} + \mbox{EOC}_{2^{-4}})/3$.}
\label{table-EOC-results-4}
\begin{center}
\begin{tabular}{|c|c|c|c|c|c|c|c|}
\hline $H$      & $k(H)$
& $\| e^{0,n} \|_{L^2(\Omega)}^{\mbox{\tiny rel}}$
& $\| e^{\ms,n} \|_{L^2(\Omega)}^{\mbox{\tiny rel}}$
& $\| e^{\ms,n} \|_{H^1(\Omega)}^{\mbox{\tiny rel}}$
& $\| \partial_t e^{\ms,n} \|_{L^2(\Omega)}^{\mbox{\tiny rel}}$
& $\| \partial_t e^{\ms,n} \|_{H^1(\Omega)}^{\mbox{\tiny rel}}$\\
\hline
\hline $2^{-2}$ & 1   & 0.2809  & 0.2598  & 0.6180 & 0.3522 & 0.9735 \\
\hline $2^{-3}$ & 2   & 0.1188  & 0.0894  & 0.4869 & 0.2473 & 0.9593 \\
\hline $2^{-4}$ & 3   & 0.0423  & 0.0289  & 0.3042 & 0.1823 & 0.9479 \\
\hline $2^{-5}$ & 3   & 0.0173  & 0.0109  & 0.2059 & 0.1229 & 0.9171 \\
\hline
\hline \multicolumn{2}{|c|}{EOC} 
                                 & 1.34      & 1.53      & 0.53     & 0.51     & 0.03 \\
\hline
\end{tabular}\end{center}}
\end{table}

The corresponding numerical results are depicted in Table \ref{table-layers-results-4}. In order to compute the experimental orders of convergence, we pick as before the relation $k=k(H):=\lfloor |\ln(H)| + 0.5 \rfloor$ to compute a suitable truncation parameter. The results are shown in Table \ref{table-EOC-results-4}. The crucial observation that we make is that Theorem \ref{apriori-lod-homogenization} seems to be indeed optimal and that Theorem \ref{apriori-semidiscrete} does not generalize to {\it not well-prepared} initial values. We observe that the $L^{\infty}(L^2)$-error converges with an order between $1$ and $1.5$, which is just what we predicted. We also note that even though the error becomes slightly smaller when we use the corrector $Q_{h,k}(u_{H,k})$, the rates of convergence remain basically the same (this is in contrast to the case of well-prepared data, where we could improve the rates with the corrector). For the convergence rates for the error in the $W^{1,\infty}(L^2)$- and $L^{\infty}(L^2)$-norms we still observe a rate of order $0.5$, which might be an indication that a relaxed version of Theorem \ref{apriori-semidiscrete} might still hold for arbitrary initial values. However, from the analytical point a corresponding argument is still missing. The main conclusion that we can draw from Model Problem 4 is that we can also numerically confirm that
the proposed multiscale method still works in the general setting of not well-prepared data. This is a significant observation that distinguishes our multiscale method from other multiscale methods for the wave equation.

\medskip
$\\$
{\bf Acknowledgements.}
This work is partially supported by the Swiss National Foundation, Grant: No $200021\_134716/1$ and by the Deutsche Forschungsgemeinschaft, DFG-Grant: OH 98/6-1. We would also like to thank the anonymous reviewers for their constructive feedback that helped us to improve the contents of this paper.

\def\cprime{$'$}

\end{document}